\documentclass[11pt]{amsart}
\usepackage{amsmath,amssymb,amsthm,latexsym,cite,cancel}
\usepackage[small]{caption}
\usepackage{graphicx,wasysym,overpic,tikz,color}
\usepackage{subfigure}
\usepackage{cite}
\usepackage[colorlinks=true,urlcolor=blue,
citecolor=red,linkcolor=blue,linktocpage,pdfpagelabels,
bookmarksnumbered,bookmarksopen]{hyperref}
\usepackage[italian,english]{babel}
\usepackage{units}
\usepackage{enumitem}
\usepackage[left=2.1cm,right=2.1cm,top=2.71cm,bottom=2.71cm]{geometry}
\usepackage[hyperpageref]{backref}
\usepackage{float,bm}

\usepackage[colorinlistoftodos]{todonotes}
\newtheorem{theorem}{Theorem}[section]
\newtheorem{definition}[theorem]{Definition}

\newtheorem{lemma}[theorem]{Lemma}

\newtheorem{corollary}[theorem]{Corollary}

\newcommand{\abs}[1]{\lvert#1\rvert}
\newcommand{\norm}[1]{\lVert#1\rVert}
\newcommand{\red}[1]{\textcolor{red}{#1}}
\newcommand{\blue}[1]{\textcolor{blue}{#1}}

\newcommand{\G}{\mathcal{G}}

\newcommand{\E}{\mathrm{E}}

\newcommand{\R}{\mathbb{R}}
\newcommand{\C}{\mathbb{C}}

\tikzstyle{nodo}=[circle,draw,fill,inner sep=0pt,minimum size=%
1.5mm]

\numberwithin{equation}{section}

\title[A large class of elliptic equations on bounded domains]{Existence and multiplicity of normalized solutions to a large class of elliptic equations on bounded domains with general boundary conditions}

\author[C.O. Alves]{Claudianor O. Alves}

\address[C.O. Alves]{\newline\indent
	Unidade Acad\^{e}mica de Matem\'{a}tica
	\newline\indent
	Universidade Federal de Campina Grande
	\newline\indent
	PB CEP:58429-900, Brazil}
\email{\href{mailto: coalves@mat.ufcg.edu.br}{coalves@mat.ufcg.edu.br}}

\author[Z. He]{Zhentao He}

\address[Z. He]{\newline\indent
	School of Mathematics
	\newline\indent
	East China University of Science and Technology
	\newline\indent
	Shanghai 200237, PR China }
\email{\href{mailto:hezhentao2001@outlook.com}{hezhentao2001@outlook.com}}

\author[C. Ji]{Chao Ji}

\address[C. Ji]{\newline\indent
	School of Mathematics
	\newline\indent
	East China University of Science and Technology
	\newline\indent
	Shanghai 200237, PR China }
\email{\href{mailto:jichao@ecust.edu.cn}{jichao@ecust.edu.cn}}


\date{\today}

\begin{document}

\begin{abstract}In this paper, by adapting the perturbation method, we study the existence and multiplicity of normalized solutions for the following  nonlinear Schr\"odinger equation
	$$
	\left\{
	\begin{array}{ll}
		-\Delta u = \lambda u + f(u)\quad & \text{in } \Omega, \\
		\mathcal{B}_{\alpha,\zeta,\gamma}u = 0 & \text{on } \partial \Omega, \\
		\int_{\Omega} |u|^2\,dx = \mu,
	\end{array}
	\right.
	\leqno{(P)^\mu_{\alpha,\zeta,\gamma}}
	$$
	where $\Omega \subset \mathbb{R}^N$ ($N \geq 1$) is a smooth bounded domain, $\mu>0$  is prescribed, $\lambda \in \mathbb{R}$ is a part of the unknown which appears  as a Lagrange multiplier, $f,g:\mathbb{R} \to \mathbb{R}$ are continuous functions satisfying some technical conditions. The boundary operator $\mathcal{B}_{\alpha,\zeta,\gamma}$ is defined by
	$$
	\mathcal{B}_{\alpha,\zeta,\gamma}u=\alpha u+\zeta \frac{\partial u}{\partial \eta }-\gamma g(u),
	$$
where $\alpha,\zeta,\gamma \in \{0,1\}$ and $\eta$ denotes the outward unit normal on $\partial\Omega$. Moreover, we highlight several further applications of our approach, including the nonlinear Schr\"{o}dinger equations with critical exponential growth in $\mathbb{R}^{2}$, the nonlinear Schr\"{o}dinger equations with magnetic fields, the biharmonic equations, and the Choquard  equations, among others.
\end{abstract}
\maketitle


\section{Introduction}
 In this paper, we will investigate the existence and multiplicity of normalized solutions to the following  nonlinear Schr\"odinger equation
	$$
	\left\{
	\begin{array}{ll}
		-\Delta u = \lambda u + f(u)\quad & \text{in } \Omega, \\
		\mathcal{B}_{\alpha,\zeta,\gamma}u = 0 & \text{on } \partial \Omega, \\
		\int_{\Omega} u^2\,dx = \mu,
	\end{array}
	\right.
	\leqno{(P)^\mu_{\alpha,\zeta,\gamma}}
	$$
	where $\Omega \subset \mathbb{R}^N$ ($N \geq 1$) is a smooth bounded domain, $\mu>0$, $\lambda \in \mathbb{R}$ is a part of the unknown which appears  as a Lagrange multiplier, $f,g:\mathbb{R} \to \mathbb{R}$ are continuous functions satisfying some technical conditions. The boundary operator $\mathcal{B}_{\alpha,\zeta,\gamma}$ is defined by
	$$
	\mathcal{B}_{\alpha,\zeta,\gamma}u=\alpha u+\zeta \frac{\partial u}{\partial \eta }-\gamma g(u),
	$$
where $\alpha,\zeta,\gamma \in \{0,1\}$ and $\eta$ denotes the outward unit normal on $\partial\Omega$.

If $\alpha=1$ and $\zeta=\gamma=0$, problem $(P)^\mu_{\alpha,\zeta,\gamma}$ becomes the following one
\begin{equation*}\label{P1}
	\left\{
	\begin{aligned}
		&-\Delta u= \lambda u + f(u) \textrm{ \ in \ } \Omega,\\
		& u \in H_{0}^{1}(\Omega), \\
		&\int_{\Omega}u^{2}dx=\mu,
	\end{aligned}
	\right.
        \leqno{(P)^\mu_{1,0,0}}
\end{equation*}
where $\Omega \subset \R^N ( N \geq 1)$  can be a smooth bounded domain or the whole $\R^N$. Problem $(P)^{\mu}_{1,0,0}$ is strongly related with the following time-dependent, nonlinear Schr\"{o}dinger equation
\[
\left\{
\begin{array}{ll}
i \dfrac{\partial \Psi}{\partial t}(t,x) = \Delta_x \Psi(t,x) + h(|\Psi(t,x)|)\Psi(t,x), & (t,x) \in \mathbb{R}^{+} \times \Omega, \\
\int_{\Omega} |\Psi(t,x)|^2\,dx = \mu &
\end{array}
\right.
\]
where $f(u)=h(|u|)u$, the prescribed mass \(\sqrt{\mu}\)  appears in nonlinear optics and the theory of Bose-Einstein condensates (see \cite{Fibich, TVZ, Zhang}). Solutions \(u\) to $(P)^{\mu}_{1,0,0}$ correspond to standing waves \(\Psi(t,x) = e^{-i\lambda t} u(x)\) of the foregoing time-dependent equation. The prescribed mass represents the power supply in nonlinear optics or the total number of particles in Bose-Einstein condensates.

If $\Omega=\R^N$, there is a vast literature concerning the existence, multiplicity and stability of normalized solutions for problem $(P)^{\mu}_{1,0,0}$. In the seminar paper \cite{CL}, for the case of $L^{2}$-subcritical, Cazenave  and Lions proved the existence of normalized ground states by the minimization method and presented a general method which enables us to prove the orbital stability of some standing waves for problem $(P)^{\mu}_{1,0,0}$. However, the $L^{2}$-supercritical problem is
totally different. In \cite{jeanjean1}, Jeanjean introduced the fibering map
$$
\tilde{I}(u,s) = \frac{e^{2s}}{2}\int_{\R^N}\abs{\nabla u}^2\, dx-\frac{1}{e^{sN}}\int_{\R^N}F(e^\frac{sN}{2}u(x))\, dx,
$$
where $F(t):=\int_0^tf(s)\,ds$. Using the scaling technique and the Poho\v{z}aev identity, the author proved that $I(u):=\int_{\R^N}\abs{\nabla u}^2\, dx-\int_{\R^N}F(u)\, dx$ on $S(\mu):=\{u \in H^1(\R^N): \int_{\R^N}u^2\, dx=\mu\}$ and  $\tilde{I}$ on $S(\mu) \times\R$ satisfy the mountain pass geometry with equal mountain pass levels, then the author constructed a bounded Palais-Smale sequence on $S(\mu)$ and obtained a radial solution to problem $(P)^{\mu}_{1,0,0}$. It is worth noting that the scaling technique plays a crucial role  in the study of normalized solutions to problem $(P)^{\mu}_{1,0,0}$ in $\R^N$. For further results  related  to problem $(P)^{\mu}_{1,0,0}$ in $\R^N$,  see \cite{A21, CCM, AJ21, AlvesThin, BartschJeanjeanSaove, Bartschmolle, BartschSaove, valerio, Bartosz, BellazziniJeanjeanLuo, CazenaveLivro, Jun, jeanjean1, JeanjeanJendrejLeVisciglia, JeanjeanLu, JeanjeanLu2020, JeanjenaLu2, JeanjeanLe, Nicola1, Nicola2} and the references therein.

If $\Omega$ is bounded, unlike the whole space $\mathbb{R}^N$, we are no longer faced with the lack of compactness, and some problems become simpler. For instance, in the $L^2$-subcritical case, the compact embeddings imply the energy functional is weakly lower semi-continuous, and then, since the energy functional is bounded from below and coercive on $S(\mu):=\{u \in H^1_0(\Omega): \int_{\Omega}u^2\, dx=\mu\}$, one can easily prove the existence of normalized ground states by the minimization method. However, in the $L^2$-supercritical case, some new difficulties arise. For example, the scaling technique---commonly used in the analysis on $\mathbb{R}^N$---is no longer applicable, and the Poho\v{z}aev identity for bounded domains has an additional term involving the boundary of the set $\Omega$, which creates further complications when  attempting  to prove the boundedness of minimizing sequences.  To address these challenges, several works have focused on the $L^2$-supercritical case. In \cite{Darion}, when $f(u)=\abs{u}^{p-2}u$ and $2 +\frac{4}{N}\leq p <2^*$, using blow-up analysis and Morse theory, Pierotti and Verzini proved that, for any $k \in \mathbb{N}^+$, problem  $(P)^{\mu}_{1,0,0}$ admits solutions having Morse index bounded above by $k$ only if $\mu$ is sufficiently small and then, they provide existence results for certain ranges of $\mu$.
Their assumptions on $f$ are different from ours. Since we are dealing with general $f \in C(\R,\R)$,  the arguments based on Morse theory and blow-up analysis do not apply in our case.

In \cite{BQZ}, using the monotonicity trick and the Pohoz\v{a}ev identity, Bartsch, Qi and Zou proved the existence of
normalized solutions $(u_{r,\mu},\lambda_{r,\mu})\in H_0^1(r\Omega)\times\R$ to the following problem
\begin{equation}\label{eqbqz}
    \left\{
\begin{array}{ll}
	-\Delta u +V(x)u = \lambda u + \abs{u}^{q-2}u + \beta\abs{u}^{p-2}u\quad & \text{in } r\Omega, \\
	u = 0 & \text{on } \partial r\Omega, \\
	\int_{r\Omega} u^2\,dx = \mu,
\end{array}
\right.
\end{equation}
and showed that $u_{r,\mu} > 0$ in $r\Omega$, $\displaystyle \limsup_{r \to \infty}\max_{x \in r \Omega}u_{r,\mu}<C_\mu$ for some positive constant $C_\mu > 0$ and $\displaystyle \limsup_{r \to \infty}\lambda_{r,\mu}<0$, where $\Omega$ is a bounded smooth star-shaped domain or $\R^N$, $N \geq 3$,  $2<p<2 +\frac{4}{N} <q <2^*$, $\beta \in \R$, and $r>0$ is large. The inequality $x\cdot \bm{n}(x)\geq 0$ on $\partial\Omega$ is crucial in the authors' argument involves Pohoz\v{a}ev identity, where $\bm{n}(x)$ denotes the outward unit normal vector at $x \in \partial\Omega$. However, for a general  smooth bounded domain $\Omega$, this inequality may not hold. Specifically, the authors proved that for any fixed $\mu >0$, there exists $r^*>0$ such that, for all $r>r^*$, problem \eqref{eqbqz} admits a solution. Compared to this existence result for the case $\beta =0$, we will establish an existence result for the general $L^2$-supercritical case  under the condition that the domain has a sufficiently small first eigenvalue. In particular, this implies that for any smooth bounded domain $\tilde{\Omega}$ and any fixed $\mu >0$, there exists $\tilde{r}>0$ such that, for all $r>\tilde{r}$, problem $(P)^{\mu}_{1,0,0}$ with $\Omega =r\tilde{\Omega}$ admits a solution. This conclusion follows from the fact that if $\tilde{\lambda}_1$ denotes the first eigenvalue of $(-\Delta,H_0^1(\tilde{\Omega}))$, then the first eigenvalue of $(-\Delta,H_0^1(r\tilde{\Omega}))$ is given by  $\frac{\tilde{\lambda}_1}{r^2}$. In \cite{Song}, for the Br\'ezis-Nirenberg problem
$$
\left\{
\begin{array}{ll}
	-\Delta u = \lambda u + \abs{u}^{2^*-2}u\quad & \text{in } \Omega, \\
	u = 0 & \text{on } \partial \Omega, \\
	\int_{\Omega}u^2\,dx = \mu,
\end{array}
\right.
$$
where $\Omega \subset \R^N(N \geq 3)$ is a smooth bounded domain, Song and Zou constructed a new type of linking within an open set of a Hilbert-Riemannian manifold below a certain energy, and proved the existence of multiple sign-changing normalized solutions. In addition, the authors also studied the Morse index of the solutions and the bifurcation points of the above problem. In particular, the authors extended their results to the cases $f(x,u)+ \abs{u}^{2^*-2}u$ (for $N\geq 3$) and $f(x,u)$ (for $N\geq 1$) without assuming that $f(x,u)$ is odd with respect to $u$ for all $x \in \bar{\Omega}$. For the latter case, they assumed that $f \in C(\bar{\Omega}\times \R,\R)$, $f(x,0)=0$, $f\not \equiv 0$ and there exists $2<p\leq q<2^*$ such that
$$
 0\leq pF(x,t)\leq f(x,t)t \leq qF(x,t).
$$
In contrast, for general nonlinearities $f\in C(\R,\R)$, we establish the existence of solutions to problem $(P)^{\mu}_{1,0,0}$ under  more general assumptions. Moreover, our approach can be naturally generalized to the problems posed in Banach spaces, such as the existence of normalized solutions to $p$-Laplacian equation on bounded domain. For more interesting  results related to normalized solutions on bounded domains, we refer the reader to \cite{AS2, BQZ, Chang, BenedettaTavaresVerzini, BenedettaTavaresVerzini2, Darion, Song} and the references therein.

In this paper, by adapting the perturbation method developed in \cite{Bu,Es}, we investigate the existence and multiplicity of normalized solutions for  problem $(P)^{\mu}_{1,0,0}$. Compared with the results described earlier,  we consider a general nonlinearity $f \in C(\R,\R)$, without any differentiability assumptions. As a consequence, several tools, for example,  the monotonicity trick, Morse theory and blow-up analysis no longer applicable. 
More importantly, the approach developed in this paper can be applied to a large class of elliptic equations when studying the existence and multiplicity of normalized solutions. To prove the main result, we first establish a non-existence result in Lemma \ref{lemnonex} for the nonlinear term $f$ satisfying ($f_1$)-($f_3$) or satisfying ($f_1$) and ($f_4$). Then,  using the information on the spectrum of $(-\Delta,H_0^1(\Omega))$, we prove existence and multiplicity of normalized solutions for problem $(P)^{\mu}_{1,0,0}$. In particular, in the Appendix,  we provide  a detailed proof of a mountain pass type theorem for $C^1$ functionals defined on open subsets, which was used in \cite{Bu} and will play a crucial role in establishing our existence results.

If $\alpha=0$, $\zeta=1$ and $\gamma=0$, problem $(P)^\mu_{\alpha,\zeta,\gamma}$ becomes the following one
\begin{equation*}\label{P2}
    \left\{
\begin{array}{ll}
	-\Delta u = \lambda u + f(u)\quad & \text{in } \Omega, \\
	\frac{\partial u}{\partial \eta} = 0 & \text{on } \partial \Omega, \\
	\int_{\Omega} u^2\,dx = \mu,
\end{array}
\right.
\leqno{(P)^\mu_{0,1,0}}
\end{equation*}
where $\Omega \subset \R^N ( N \geq 1)$ is a smooth bounded domain. A key motivation for studying problem $(P)^{\mu}_{0,1,0}$ is the ergodic Mean Field Games (MFG) system. MFG were independently introduced in seminal works by Huang, Caines, and Malham\'{e} \cite{Huang} and by Lasry and Lions \cite{Lasry}. The primary objective of MFG is to establish a framework for characterizing Nash equilibria in differential games involving an infinite number of agents that are indistinguishable from one another. For more details, we refer reader to \cite{Chang,Cirant1, Cirant2, Santambrogio} and the references therein. Such equilibria can be characterized by a system satisfying normalization in $L^1(\Omega)$ as follows
\begin{equation}\label{eqmfg}
    \begin{cases}
        -\Delta v + H(\nabla v) + \lambda = h(m(x)) & \text{in } \Omega, \\
        -\Delta m - \text{div}(m\nabla H(\nabla v)) = 0 & \text{in } \Omega, \\
        \dfrac{\partial v}{\partial \eta} = 0,\quad
        \dfrac{\partial m}{\partial \eta} + m\nabla H(\nabla v) \cdot v = 0 & \text{on } \partial \Omega, \\
        \displaystyle\int_{\Omega} m  dx = 1,\quad
        \int_{\Omega} v  dx = 0,
    \end{cases}
\end{equation}
From the point of view of PDE, problem \eqref{eqmfg} is an elliptic system that combines a Kolmogorov equation for $m$ and a Hamilton-Jacobi-Bellman equation for $v$. The Neumann boundary conditions are based on the assumption that agents' trajectories are restricted to $\Omega$ by bouncing off the boundary in a normal direction. For the quadratic Hamilton case $ H(\nabla v) = |\nabla v|^2 $, by using a Hopf-Cole transformation $ \phi = e^{-\nu}/\int e^{-\nu} = \sqrt{m} $, \eqref{eqmfg} is reduced to \begin{equation}\label{eqmfgp2}
    \left\{
\begin{array}{ll}
	-\Delta \phi = \lambda \phi + h(\phi^2)\phi\quad & \text{in } \Omega, \\
	\frac{\partial \phi}{\partial \eta} = 0 & \text{on } \partial \Omega, \\
	\int_{\Omega} \phi^2\,dx = 1.
\end{array}
\right.
\end{equation}
Problem \eqref{eqmfgp2} is equivalent to problem $(P)^{\mu}_{0,1,0}$ by taking a simple transformation $\phi =\frac{1}{\sqrt{c}}u$ and $f(u)=h(\frac{u^2}{c})u$.

In \cite{Pellacci}, Pellacci, Pistoia, Vaira and Verzini investigated normalized solutions of problem $(P)^{\mu}_{1,0,0}$ and $(P)^{\mu}_{0,1,0}$, focusing on the concentration behavior of solutions at specific points of $\Omega$ as the prescribed mass $\mu$ varies. Specifically, for the Neumann problem $(P)^{\mu}_{0,1,0}$ with $f(u) = \vert u\vert^{p-2}u$ and $2+\frac{4}{N}<p<2^*$, they employed the Lyapunov-Schmidt reduction method to prove the existence of positive normalized solutions for $\mu \in (0, \mu_0)$, where $\mu_0 > 0$. The solutions concentrate at a point $\xi_0 \in \overline{\Omega}$ as $\mu \to 0$, where $\xi_0$ is either a non-degenerate critical point of the mean curvature $H$ of the boundary $\partial \Omega$, or the maximum point of the distance function from $\partial \Omega$. In \cite{Chang}, using the monotonicity trick, Morse theory and blow-up analysis, Chang, R\v{a}dulescu and Zhang proved the existence of normalized solutions to problem $(P)^{\mu}_{0,1,0}$  for $\mu \in (0, \mu_0)$ with $\mu_0 > 0$. They assumed that $\Omega \subset \R^N$ ($N \geq 3$) is a smooth bounded domain and that $f\in C^1(\R,\R)$ satisfies
$$
\displaystyle \lim_{|t| \to 0} \frac{f(t)}{t} = 0,
$$
there exist constants $p \in (2,2^*)$ and  $a_0 > 0$ such that
$$
          \lim_{|t| \to \infty} \frac{f(t)}{|t|^{p-2}t} = a_0,
$$
and there exist constants $\mu \geq a_0(p-1)$ and $M > 0$ such that
$$
f'(t)\geq \mu |t|^{p-2}, \quad \forall |t| \geq M.
$$
It is worth noting that their assumptions on $f$ differ substantially from ours. Since we dealing with general nonlinearities $f \in C(\R,\R)$ without assuming differentiability, the techniques employed in  \cite{Chang}, namely, the monotonicity trick, Morse theory, and blow-up analysis, are not applicable in our setting. Moreover,  to the best of our knowledge, there are no existing multiplicity results for problem $(P)^{\mu}_{0,1,0}$, even in the case where  $f(u)=\abs{u}^{p-2}u$ with $2+\frac{N}{4}<p<2^*$.

To present the main results of this paper, we assume throughout this paper that the nonlinearity $f \in  C(\R, \R)$ satisfies the following assumptions:
\begin{enumerate}[align=left,labelsep=0pt]
    \item[($f_1$)] There exist constants $2<p<2^*$ and $0\leq K_2<\lambda_{1,0,0}$ and $K_{p} \geq 0$ such that
    $$
    \abs{f(t)} \leq K_2\abs{t} + K_{p}\abs{t}^{p-1}, \quad \forall t \in \mathbb{R},
    $$
    where $2^* = \frac{2N}{N-2}$ if $N \geq 3$ and $2^* = +\infty$ if $N = 1, 2$.
    \item[($f_2$)]  $f(t)t\geq 0$ and $t\mapsto \frac{f(t)}{\abs{t}}$ is non-decreasing on $(-\infty,0)\cup(0,+\infty)$.

    \item[($f_3$)] $\displaystyle \lim_{\abs{t}\to +\infty}\frac{F(t)}{t^2}=+\infty$, where $F(t):=\int_0^tf(s)\,ds$.
\end{enumerate}
\begin{enumerate}[align=left,labelsep=0pt]
	\item[($g_1$)] There exist constants $2<l<2^{\hexstar}=\frac{2(N-1)}{N-2}$ and  $0\leq K_2^g<\frac{\tilde{\lambda}_{1,1,1}}{4}$ and $K_{l} \geq 0$ such that
	$$
	\abs{g(t)} \leq K_2^g\abs{t} + K_{l}\abs{t}^{l-1}, \quad \forall t \in \mathbb{R},
	$$
	where $\tilde{\lambda}_{1,1,1}$ is defined in \eqref{eqtilde}, $2^{\hexstar} = \frac{2(N-1)}{N-2}$ if $N \geq 3$ and $2^{\hexstar} = +\infty$ if $N =2$.
	
\end{enumerate}
We would like to  emphasize that a broad class of nonlinearities  $f$ satisfying ($f_1$)-($f_3$).  For example,
$$f(t)=\abs{t}^{s-2}t\quad  \text{ with } 2<s<2^*,$$
$$f(t)=\abs{t}^{s_1-2}t + \abs{t}^{s_2-2}t\quad  \text{ with } 2<s_1,s_2<2^*,$$
and
$$f(t)=\abs{t}^{s_1-2}t + \abs{t}^{s_2-2}t\ln(1+\abs{t})\quad  \text{ with } 2<s_1,s_2<2^*.$$
Moreover, by \cite[Remark 1.1]{Zhong}, we can easily construct a function $f$ satisfies ($f_1$)-($f_3$) but the mapping $t\mapsto \frac{f(t)}{\abs{t}}$ is not strictly increasing on $(-\infty,0)\cup(0,+\infty)$. For instance, consider the piecewise-defined function:
$$
f(t)=
\begin{cases}
    \abs{t}^{s_1-2}t &\text{ if } t\in [-1,1],\\
    t &\text{ if } t\in [-2,-1)\cup(1,2],\\
    \abs{t-1}^{s_2-2}t &\text{ if } t\in \R \backslash[-2,2],
\end{cases}
$$
with $2<s_1,s_2<2^*$.

In the study of problem $(P)^{\mu}_{1,0,0}$,  let  $\lambda_{1,0,0}$ denote the first eigenvalue of the operator $(-\Delta, H_{0}^{1}(\Omega))$. Our first result is stated as follows:

\begin{theorem} \label{th2}
 Suppose that ($f_1$)-($f_3$) hold. Then, there exists $\mu_0>0$ depending only on $\Omega$ and $f$ such that, for all $0<\mu<\mu_0$, $(P)^{\mu}_{1,0,0}$ has a solution $(u,\lambda) \in H^1_0(\Omega) \times [0,\lambda_{1,0,0}]$.
\end{theorem}

Our next result is a version of the previous theorem under the alternative assumption that the nonlinearity  $f$ satisfies the well-known Ambrosetti-Rabinowitz condition, namely
\begin{enumerate}[align=left,labelsep=0pt]
    \item[($f_4$)] There exists $q>2$ such that $0 \leq qF(t)\leq f(t)t$ for all $t \in \mathbb{R}$.
\end{enumerate}
\begin{theorem}\label{th3}
    Suppose that ($f_1$) and ($f_4$) hold. If $$0<\mu \leq
  \begin{cases}
    \left(\frac{\lambda_{1,0,0}-K_2}{K_pC_{p,N}}\right)^{\frac{2}{p-2}}\lambda_{1,0,0}^{-\frac{N}{2}}  &\text{ if } 2<p\leq 2+\frac{4}{N},\\
    \left(\frac{\lambda_{1,0,0}-K_2}{K_pC_{p,N}}\right)^{\frac{2}{p-2}}\left(\frac{q}{q-2}\right)^{\frac{2}{p-2}-\frac{N}{2}}\lambda_{1,0,0}^{-\frac{N}{2}} &\text{ if }  2+\frac{4}{N}<p<2^*,
\end{cases}$$
holds, where $C_{p,N}$ is defined in \eqref{eqgn}, then $(P)^{\mu}_{1,0,0}$ has a solution $(u,\lambda) \in H^1_0(\Omega) \times [0,\lambda_{1,0,0}]$.
\end{theorem}
From ($f_4$),  we know that for any $R>0$ such that $F(R)>0$, one gets
$$
F(t)\leq \frac{F(R)}{R^q} \abs{t}^q, \quad \forall t\in [-R,R],\quad \text{and}\quad F(t)\geq \frac{F(R)}{R^q} \abs{t}^q, \quad \forall t\in (-\infty,-R)\cup(R,+\infty).
$$
Thus, if we further assume that ($f_1$) holds and $ f\not \equiv 0$, then $q\leq p$. Then, as an immediate corollary of Theorem \ref{th3}, we have
\begin{corollary}\label{co1}
For any fixed $\mu>0$, the following two existence results depending on $\Omega$ hold:
  \begin{enumerate}[label=\rm(\roman*)]
    \item Suppose that ($f_1$) and ($f_4$) hold, and suppose that $K_2=0$ and $2+\frac{4}{N}<p<2^*$. If $\Omega$ satisfies $0<\lambda_{1,0,0}\leq \frac{q-2}{q}\mu^{\frac{2(p-2)}{4+2N-Np}}(K_pC_{p,N})^
    \frac{4}{4+2N-Np}$, then $(P)^{\mu}_{1,0,0}$ has a solution $(u,\lambda) \in H^1_0(\Omega) \times [0,\lambda_{1,0,0}]$.
    \item Suppose that ($f_4$) hold, and suppose $|f(t)|\leq a\left(\abs{t}^{p'-1} +\abs{t}^{p-1}\right)$ with $2+\frac{4}{N}<p'<p<2^*$ for some positive constant $a>0$. There exists $\lambda^*$ depending on $a$, $p$, $p'$, $q$, $N$ and $\mu$ such that, if $\Omega$ satisfies $0<\lambda_{1,0,0}<\lambda^*$, then $(P)^{\mu}_{1,0,0}$ has a solution $(u,\lambda) \in H^1_0(\Omega) \times [0,\lambda_{1,0,0}]$.
\end{enumerate}
\end{corollary}

By using the Fountain theorem, we can prove the following multiplicity results.
\begin{theorem}\label{th4}
 Suppose that ($f_1$)-($f_3$) hold or ($f_1$) and ($f_4$) hold, and suppose that $f(t)=-f(-t)$. For any $m \in \mathbb{N}^+$, there exists $\mu_{m}^* > 0$ depending on $\Omega$ and $f$, such that, for any $0 < \mu < \mu^*_{m}$,  $(P)^{\mu}_{1,0,0}$  has at least $m$ nontrivial solutions $(u_1,\bar{\lambda}_1), (u_2,\bar{\lambda}_2),\cdots,(u_m,\bar{\lambda}_m) \in H^{1}_{0}(\Omega)\times[0,+\infty)$.
\end{theorem}

Moreover, when ($f_2$) holds, we prove that the solutions $u$ obtained in Theorem \ref{th2} and \ref{th3} are minimizers of $E$ constrained on $\mathcal{N}^+(\mu)$, defined in \eqref{eqdefn+}. Furthermore, when $\Omega$ is star-shaped, $N \geq 3$ and $f$ satisfies ($f_1$), ($f_2$) and ($f_4$), we apply the perturbation method to establish the existence of a ground state $(u,\lambda)$ to problem $(P)^{\mu}_{1,0,0}$, see Definition \ref{defgs}, and we also  show that $u$ is a local minimizers of energy functional $E$ constrained on $S_{1,0,0}(\mu):=\{u \in H^1_0(\Omega): \int_{\Omega}u^2\, dx=\mu\}$. We note that the existence results of ground state for problem $(P)^{\mu}_{1,0,0}$ is not new, see \cite{Ji}. However, the author in \cite{Ji} used the minimization method to obtain the ground states and the author's assumptions on $f$ are different from ours. Finally, similar to Corollary \ref{co1}, we also establish in Corollary \ref{co2} the existence of ground state solutions for problem $(P)^{\mu}_{1,0,0}$ under the condition that the domain the domain has a sufficiently small first eigenvalue.

For the case where $\alpha=\zeta=\gamma=1$, we are working with the following problem
$$
\left\{
\begin{array}{ll}
	-\Delta u = \lambda u + f(u)\quad & \text{in } \Omega, \\
	\frac{\partial u}{\partial \eta } = -u+g(u) & \text{on } \partial \Omega, \\
	\int_{\Omega} u^2\,dx = \mu,
\end{array}
\right.
\leqno{(P)^{\mu}_{1,1,1}}
$$
where $\Omega \subset \R^N ( N \geq 2)$ is a smooth bounded domain. In the study of the above problem, we define  
\begin{equation}\label{eqhat}
    \hat{\lambda}_{1,1,1}:=\inf\left\{\int_{\Omega}|\nabla u|^2\,dx+\int_{\partial \Omega}|u|^2\,d\sigma\,:\,\int_{\Omega}|u|^2\,dx=1\right\}
\end{equation}
and
\begin{equation}\label{eqtilde}
    \tilde{\lambda}_{1,1,1}:=\inf\left\{\int_{\Omega}|\nabla u|^2\,dx+\int_{\partial \Omega}|u|^2\,d\sigma\,:\,\int_{\partial \Omega}|u|^2\,d\sigma=1\right\}\blue{=1.}
\end{equation}
 To study $(P)^{\mu}_{1,1,1}$, we need the following assumptions on $f$ and $g$:
\begin{enumerate}[align=left,labelsep=0pt]
    \item[($f_1'$)] There exist constants $2<p<2^*$ and $0\leq K_2<\frac{\hat{\lambda}_{1,1,1}}{4}$ and $K_{p} \geq 0$ such that
    $$
    \abs{f(x,t)} \leq K_2\abs{t} + K_{p}\abs{t}^{p-1}, \quad \forall t \in \mathbb{R},
    $$
    where $2^* = \frac{2N}{N-2}$ if $N \geq 3$ and $2^* = +\infty$ if $N = 1, 2$.
    \item[($g_2$)] There exists $q>2$ such that $0 \leq qG(t)\leq g(t)t$, where $G(t):=\int_0^tg(s)\,ds$.
\end{enumerate}
\begin{theorem} \label{th2.3}
Suppose that ($f_1'$), ($f_4$), ($g_1$) and ($g_2$) hold with $N \geq 2$. Then, there exists $\mu^{**}>0$ depending on $\Omega$, $f$ and $g$ such that, for all $0<\mu<\mu^{**}$, $(P)^{\mu}_{1,1,1}$ admits a solution $(u,\lambda) \in H^1(\Omega) \times  [0, \hat{\lambda}_{1,1,1}]$.
\end{theorem}

By using the Fountain theorem, we can prove the following multiplicity results.
\begin{theorem}\label{th4.3}
Suppose that ($f_1'$), ($f_4$), ($g_1$) and ($g_2$) hold with $N \geq 2$, and suppose that $f(t)=-f(-t)$ and $g(t)=-g(-t)$. For any $m \in \mathbb{N}^+$, there exists $\mu^{**}_{m} > 0$ depending on $\Omega$, $f$ and $g$, such that, for any $0 < \mu < \mu^{**}_{m}$,  $(P)^{\mu}_{1,1,1}$  has at least $m$ nontrivial solutions $(u_1,\bar{\lambda}_1), (u_2,\bar{\lambda}_2),\cdots,(u_m,\bar{\lambda}_m) \in H^{1}(\Omega)\times[0,+\infty)$.
\end{theorem}

The main difficulty in proving Theorems \ref{th2.3} and \ref{th4.3} lies in the fact that  that it is unclear whether the following  Gagliardo-Nirenberg type inequality holds for each $l \in (2, 2^{\hexstar})$:
\begin{equation*}
	\|u\|_{L^l(\partial \Omega)}^l \leq C_{l,N} \|u\|_{L^2(\Omega)}^{(1 - \beta_l)l} \|u\|_{H^{1}(\Omega)}^{\beta_l l}, \quad \forall u \in  H^1(\Omega),
\end{equation*}
which plays a crucial role in proving  that the energy functional is bounded from below and coercive on $S_{1,1,1}(\mu):=\{u \in H^1(\Omega): \int_{\Omega}\abs{u}^2\, dx=\mu\}$ in the  $L^2$-subcritical and $L^2$-critical cases, or possesses the mountain pass geometry on $S_{1,1,1}(\mu)$ in the $L^2$-supercritical case. In our approach, we shall avoid using on the above inequality. To the best of our knowledge, there are no existence and multiplicity results associated with problem $(P)^{\mu}_{1,1,1}$ in the literature.

For the case where $\zeta=1$ and $\alpha=\gamma=0$, we are working with the problem below
$$
\left\{
\begin{array}{ll}
	-\Delta u = \lambda u + f(u)\quad & \text{in } \Omega, \\
	\frac{\partial u}{\partial \eta } = 0 & \text{on } \partial \Omega, \\
	\int_{\Omega} u^2\,dx = \mu.
\end{array}
\right.
\leqno{(P)^{\mu}_{0,1,0}}
$$
In the study of problem above,  let us denote by $\lambda_{0,1,0}=1$ the first eigenvalue of the operator $(-\Delta+I, H^{1}(\Omega))$. We need the following assumption.
\begin{enumerate}[align=left,labelsep=0pt]
	\item[($f_1''$)] There exist constants $2<p<2^*$ and $0\leq K_2<\lambda_{0,1,0}$ and $K_{p} \geq 0$ such that
	$$
	\abs{f(x,t)} \leq K_2\abs{t} + K_{p}\abs{t}^{p-1}
	$$
	where $2^* = \frac{2N}{N-2}$ if $N \geq 3$ and $2^* = +\infty$ if $N = 1, 2$.
\end{enumerate}

\begin{theorem}\label{th4.2}
	Suppose that  ($f_1''$), ($f_2$) and ($f_3$) hold or ($f_1''$) and ($f_4$) hold, and suppose that $f(t)=-f(-t)$. For any $m \in \mathbb{N}^+$, there exists $\mu_{m}^{***}> 0$ depending on $\Omega$ and $f$, such that, for any $0 < \mu < \mu_{m}^{***}$,  $(P)^{\mu}_{0,1,0}$  has at least $m-1$ nontrivial solutions $(u_1,\bar{\lambda}_1), (u_2,\bar{\lambda}_2),\cdots,(u_{m-1},\bar{\lambda}_{m-1}) \in H^{1}(\Omega)\times[-1,+\infty)$ and, for all $1\leq i\leq m-1$, $u_i$ is not a constant function.
\end{theorem}


Finally, we would like to emphasize that our perturbation method is applicable to a broader class of problems. It does not require any additional geometric condition on the domain and can be applied to other interesting classes of problems, such as:
\begin{itemize}
	\item[(a)] The nonlinear Schr\"{o}dinger equations with exponential critical growth in $\mathbb{R}^2$;
	\item[(b)] The nonlinear Schr\"{o}dinger equations with magnetic fields;
	\item[(c)] Bi-harmonic equations;
	\item[(d)] Choquard equations.
\end{itemize}
We refer the reader to  Section \ref{secfinalremarks} for more details.

The rest of the paper is organized as follows. In Section \ref{secp100}, we present a detailed proof of the existence and multiplicity results for problem $(P)^{\mu}_{1,0,0}$. In Section \ref{secp111} and \ref{secp010}, we provide sketches of the proofs of the existence and multiplicity results for problem $(P)^{\mu}_{1,1,1}$ and $(P)^{\mu}_{0,1,0}$, respectively. In Section \ref{secgs}, we characterize the nature of solutions obtained by perturbation method and prove the existence of ground states to $(P)^{\mu}_{1,0,0}$ when $\Omega$ is star shaped and $f$ satisfies ($f_1$), ($f_2$) and ($f_4$). In Section \ref{secfinalremarks}, we point out some problems that our approach can be applied.  In Section \ref{appa}, we give a proof of the  mountain pass type theorem for $C^1$ functionals defined on open subsets, as used in \cite{Bu,Es}.
\section{The problem \texorpdfstring{$(P)^\mu_{1,0,0}$}{Lg}}\label{secp100}
In this section, we endow $H_0^1(\Omega)$ with the standard norm
$$
\|u\|:=\left(\int_{\Omega}|\nabla u|^2\,dx\right)^{\frac{1}{2}},
$$
which induces the scalar product
$$
(u,v):=\int_{\Omega}\nabla u \nabla v\,dx.
$$
For $L^p(\Omega)$ spaces, we use the usual norm denoted by $\norm{u}_p$. The scalar product induced by $L^2(\Omega)$-norm is denoted by
$$
(u,v)_2:=\int_{\Omega} u v\,dx.
$$

We consider the energy functional $E: H^1_0(\Omega) \to \mathbb{R}$ associated with problem  $(P)^\mu_{1,0,0}$, defined by
$$
E(u) := \frac{1}{2}\norm{u}^{2}-\Psi(u),
$$
where
$$
\Psi(u)=\int_\Omega  F(u)\,dx.
$$
\label{sec:preliminaries}
\subsection{The Gagliardo-Nirenberg inequalities}
\label{subsecineq}
Since $H_0^1(\Omega)\subset H(\R^N)$, by the Gagliardo-Nirenberg inequality for  $\R^N$ (see \cite[Theorem 1.1]{Fio}),  for any  $p \in [2, 2^*)$, there exists a constant $C_{p,N}>0$, depending only on  $p$ and $N$, such that
\begin{equation}\label{eqgn}
    \|u\|_{p}^p \leq C_{p,N} \|u\|_{2}^{(1 - \beta_p)p} \|\nabla u\|_{2}^{\beta_p p} \quad \forall u \in  H_0^1(\Omega),
\end{equation}
where $\beta_p = N \left( \frac{1}{2} - \frac{1}{p} \right)$.
\subsection{The perturbation functional}\label{subsecpf}
We define the following perturbation functional
$$
E_{r, \mu}(u) := \frac{1}{2}\norm{u}^{2}-\Psi(u)-H_{r,\mu}(u),\; u \in U_{\mu},
$$
where $U_{\mu} := \{u \in H_0^1(\Omega): \norm{u}_2^2<\mu\}$, and $H_{r, \mu}(u)$ is a penalization term defined by
$$
H_{r, \mu}(u) := f_{r}\left(\frac{\norm{u}_2^2}{\mu}\right) \text { with } f_{r}(s) := \frac{s^{r}}{1-s} \; \quad \forall s \in [0,1),
$$
and $r>1$ is a parameter that will be chosen large enough. A straightforward computation gives
\begin{equation}\label{eqfr's}
    f_{r}^{\prime}(s)=\frac{r s^{r-1}}{1-s}+\frac{s^{r}}{(1-s)^{2}}>\frac{r}{s} f_{r}(s)>0, \; \quad \forall s \in (0,1)
\end{equation}
and
\begin{equation} \label{eqfr''s}
	f''_r(s)=\frac{r(r-1)s^{r-2}}{1-s}+\frac{rs^{r-1}}{(1-s)^2}+\frac{rs^{r-1}}{(1-s)^2}+\frac{2s^{r}}{(1-s)^3}>0,\; \quad \forall s \in (0,1),
\end{equation}
from where it follows that $f'_r$ is increasing for $s \in [0,1)$. Fixing $h_r(s)=f'_r(s)s-f_r(s)$ for $s \in [0,1)$, the information above yields
\begin{equation} \label{eqhrs}
	h'_r(s)=f''_r(s)s>0, \; \quad \forall s \in (0,1),
\end{equation}
showing that $h_r$ is an increasing function for $s \in [0,1)$.

It is standard to prove that $E_{r, \mu} \in C^1(U_{\mu},\R)$ and for any $u \in U_{\mu}$ and $v \in H_0^1(\Omega)$, one has
$$
\langle E_{r, \mu}'(u),v\rangle = (u,v)-\int_\Omega f(x,u)v \,dx-\frac{2}{\mu}f'_r\left(\frac{\norm{u}_2^2}{\mu}\right)(u,v)_2.
$$

Let $\beta \in C^\infty(\mathbb{R},\mathbb{R})$ be such that
$$
\begin{cases}
    \beta \equiv -1 &\text{ on } (-\infty,-1),\\
    \beta (t)=t, &\forall t\geq 0,\\
    \beta (t)\leq 0,  &\forall t\leq  0.
\end{cases}
$$
We define a new functional $J_{r,\mu}: H^1_0(\Omega) \to \mathbb{R}$ by
\begin{equation} \label{Jrmu}
J_{r,\mu}(u) =
\begin{cases}
    \beta(E_{r,\mu}(u)) &\text{ if } \norm{u}_2^2<\mu,\\
    -1  &otherwise.
\end{cases}
\end{equation}
By Lemma \ref{lemc1}, $J_{r,\mu}\in C^1(H^1_0(\Omega),\R)$. Moreover, if $u$ is a critical point of $J_{r,\mu}$ with $J_{r,\mu}(u) \geq 0$, then $u$ is also a critical point of $E_{r,\mu}$ at the same energy level.  The same conclusion holds for Cerami sequences. Therefore, instead of $E_{r,\mu}$, we can look for positive min-max levels of $J_{r,\mu}$, which is more convenient, since  $J_{r,\mu}$ is defined on $H^1_0(\Omega)$, and satisfies $J_{r,\mu}=-1$ on $\partial U_\mu$.
\subsection{Compactness of Cerami sequences}\label{seccompact}
In this subsection, we always assume that ($f_1$)-($f_3$) hold or ($f_1$) and ($f_4$) hold. Fix $\mu >0$. Assume that for any $r >1$ sufficiently large, there exists a sequence $\{u_{n,r}\}_{n\geq 1}\subset H^1_0(\Omega)$ such that $\{u_{n,r}\}_{n\geq 1}$ is a Cerami sequence of $E_{r,\mu}$ at level $c_r>0$, i.e., $\{u_{n,r}\}_{n\geq 1}$ satisfies
\begin{equation} \label{ceramisequence}
E_{r,\mu}(u_{n,r}) \rightarrow c_r \text{ and } (1+\norm{u_{n,r}})\norm{E_{r, \mu}^{\prime}(u_{n,r})} \rightarrow 0 \quad \text{ as } n \to +\infty.
\end{equation}

Moreover, we assume that $r \mapsto c_r$ is non-decreasing.
Define
$$
\lambda_{n,r} := \frac{2}{\mu}f'_r\left(\frac{\norm{u_{n,r}}_2^2}{\mu}\right).
$$
\begin{lemma}\label{lembounded}
    For $r>1$ sufficiently large, $\{u_{n,r}\}_{n\geq 1}$ is bounded in $H^1_0(\Omega)$.
\end{lemma}
\begin{proof}
We first assume that ($f_1$)-($f_3$) hold. Let $t_n \in [0, 1]$ be such that $E_{r,\mu}(t_n u_{n,r}) = \max_{t \in [0, 1]} E_{r,\mu}(tu_{n,r})$. We claim that $\{E_{r,\mu}(t_n u_{n,r})\}_{n \geq 1}$ is bounded from above. In fact, if either $t_n = 0$ or $t_n = 1$, we are done. Thereby, we can assume $t_n \in (0, 1)$, and so, $E_{r,\mu}'(t_n u_{n,r}) u_{n,r} = 0$. From this,
$$
\begin{aligned}
    2E_{r,\mu}(t_n u_{n,r}) =& 2E_{r,\mu}(t_n u_{n,r}) - E_{r,\mu}'(t_n u_{n,r}) t_n u_{n,r} \\
    =& \int_{\Omega} \left[f(t_n u_{n,r})t_n u_{n,r}-2F(t_n u_{n,r})\right]\,dx \\
    &+ 2f'_r\left(\frac{t_n^2\norm{u_{n,r}}_2^2}{\mu}\right)\frac{t_n^2\norm{u_{n,r}}^2_{2}}{\mu} -2f_r\left(\frac{t_n^2 \norm{u_{n,r}}_2^2}{\mu}\right).
\end{aligned}
$$
The condition ($f_2$) gives that $s \mapsto f(s)s-2F(s)$ is non-decreasing on $(0,+\infty)$ and non-increasing on $(-\infty,0)$. By \eqref{eqhrs}, we know that  $s \mapsto h_r(s)=f_r'(s)s-f(s)$ is increasing on $[0,1)$, and thus,
$$
\begin{aligned}
2E_{r,\mu}(t_n u_{n,r}) &\leq \int_{\Omega}\left[f(u_{n,r})u_{n,r}-2F(u_{n,r})\right]\,dx + 2f'_r\left(\frac{\norm{u_{n,r}}_2^2}{\mu}\right)\frac{\norm{u_{n,r}}^2_{2}}{\mu} -2f_r\left(\frac{ \norm{u_{n,r}}_2^2}{\mu}\right)\\
&= 2E_{r,\mu}(u_{n,r}) - E_{r,\mu}'(u_{n,r}) u_{n,r} = 2E_{r,\mu}(u_{n,r}) + o_n(1).
\end{aligned}
$$
Since $\{E_{r,\mu}(u_{n,r})\}_{n \geq 1}$ converges, it follows that $\{E_{r,\mu}(t_n u_{n,r})\}_{n \geq 1}$ is bounded from above.

To prove the boundedness of $\{u_{n,r}\}_{n\geq 1}$, we argue by contradiction. Suppose, on the contrary,  that $\norm{u_{n,r}}\to+\infty$ as $n \to +\infty$. Let $v_n=u_{n,r}/\norm{u_{n,r}}$, then $\norm{v_n}=1$. Then, up to a subsequence, there exists $v \in H^1_0(\Omega)$ such that $v_n \rightharpoonup v$. Next, we will show that $v = 0$. Otherwise, for $x\in \Omega_*=\{y\in\Omega:v(y)\neq 0\}$, we have $\displaystyle \lim_{n\to\infty}|u_{n,r}(x)|=\infty$. Hence, it follows from ($f_2$)-($f_3$) and Fatou's lemma that
\begin{equation*}
    \begin{aligned}
0 =\lim_{n\to\infty}\frac{c_r+o_n(1)}{\norm{u_{n,r}}^2}&=\lim_{n\to\infty}\frac{E_{r,\mu}(u_{n,r})}{\norm{u_{n,r}}^2} \\
&\leq \lim_{n\to\infty}\left[\frac{1}{2}-\int_{\Omega_*}\frac{F(u_{n,r})}{(u_{n,r})^2}v_n^2\,dx\right] \\
&\leqslant\frac{1}{2}-\liminf_{n\to\infty}\int_{\Omega_*}\frac{F(u_{n,r})}{(u_{n,r})^2}v_n^2\,dx
\leqslant\frac{1}{2}-\int_{\Omega_*}\liminf_{n\to\infty}\frac{F(u_{n,r})}{(u_{n,r})^2}v_n^2\,dx
=-\infty.
\end{aligned}
\end{equation*}
This contradiction shows that $v=0$.

Notice that for each $B > 0$, one has $\frac{B}{\|u_{n,r}\|} \in [0, 1]$ and $\frac{B^2\norm{v_{n}}_2^2}{\mu}= \frac{B^2}{\norm{u_{n,r}}^2}\in [0,1)$ for $n$ sufficiently large. Thus,
$$E_{r,\mu}(t_n u_{n,r}) \geq E_{r,\mu} \left( \frac{B}{\|u_{n,r}\|} u_{n,r} \right) = E_{r,\mu}(B v_n) = \frac{B^2}{2} - \int_{\Omega} F(B v_n)- f_r\left(\frac{ B^2\norm{v_{n}}_2^2}{\mu}\right).$$
Since the embedding $H^1_0(\Omega)\hookrightarrow L^p(\Omega)$ is compact for all $2\leq p<2^*$, we have
$$\liminf_{n \to \infty} E_{r,\mu}(t_n u_{n,r}) \geq \frac{B^2}{2} \quad \text{for every } B > 0,$$
which is a contradiction showing that $\{u_{n,r}\}_{n \geq 1}$ is bounded in $H^1_0(\Omega)$, once that $\{E_{r,\mu}(t_n u_{n,r})\}_{n \geq 1}$ is bounded from above.

Now, we assume that ($f_1$) and ($f_4$) hold. Then, by choosing $r>\frac{q}{2}$, ($f_4$) combined with \eqref{eqfr's} leads to
    $$
    \begin{aligned}
            \norm{u_{n,r}}^2 &\leq \frac{2qE(u_{n,r})-2\langle E'(u_{n,r}),u_{n,r}\rangle}{q-2}\\
            &\leq \frac{2qE(u_{n,r})-2\langle E'(u_{n,r}),u_{n,r}\rangle+4f'_r\left(\frac{\norm{u_{n,r}}_2^2}{\mu}\right)\frac{\norm{u_{n,r}}^2_{2}}{\mu} -2qf_r\left(\frac{ \norm{u_{n,r}}_2^2}{\mu}\right)}{q-2}\\
            &=\frac{2qE_{r,\mu}(u_{n,r})-2\langle E_{r,\mu}'(u_{n,r}),u_{n,r}\rangle}{q-2}\\
            &=\frac{2qc_r}{q-2}+o_n(1).
    \end{aligned}
    $$
This complete the proof of Lemma \ref{lembounded}.

\end{proof}
\begin{lemma}\label{lemlambdanr}
If $c_\infty:=\lim\limits_{r\to+\infty} c_r <+\infty$, then
$$
\lambda_{r}:= \limsup_{n \rightarrow+ \infty} \lambda_{n,r} < \infty
$$
and
$$
\limsup_{r \rightarrow+ \infty}\lambda_{r} \leq \frac{2c_\infty}{\mu}.
$$
\end{lemma}
\begin{proof}
 First, by assumptions $r \mapsto c_r$ is non-decreasing and  $c_{r_0} >0$ for some $r_0>1$ sufficiently large, we have $c_\infty \geq c_{r_0} >0$. Since $f_{r}^{\prime}(0)=f_{r}(0)=0$ and $h_r(s)=f_{r}^{\prime}(s) s-f_{r}(s) \to +\infty$ as $s \to 1^{-}$, by continuity of $s \mapsto f_{r}^{\prime}(s) s-f_{r}(s)$, there exists $\xi_{r} \in (0,1)$ such that
$$
c_{\infty}=f_{r}^{\prime}(\xi_{r})\xi_{r}-f_{r}(\xi_{r}).
$$
We now claim that $\xi_{r}\to 1^{-}$ as $r \to +\infty$. Argue by contradiction, suppose that there exists a subsequence $\{r_{n}\}$ such that $r_{n} \nearrow +\infty$ (monotone sequences) and $\displaystyle \xi_{r_n}\to \xi$ as $n \to +\infty$ with $\xi \in [0,1)$. Then, by the definition of $f_{r}$,
$$c_{\infty}=f_{r_n}^{\prime}(\xi_{r_n})\xi_{r_n}-f_{r_n}(\xi_{r_n})=\frac{(r_n-1) \xi_{r_n}^{r_n}}{1-\xi_{r_n}}+\frac{\xi_{r_n}^{r_n+1}}{(1-\xi_{r_n})^{2}} \to 0\quad \text{as } n\to +\infty,$$
which is a contradiction. Then, $\xi_{r}\to 1^{-}$. Thus, by \eqref{eqfr's},
\begin{equation}\label{eqsupfr}
 f_{r}(\xi_{r}) \to 0\quad \text{and}\quad  f_{r}^{\prime}(\xi_{r}) \to c_{\infty}\quad \text{as } r\to +\infty.
\end{equation}
Thanks to \eqref{ceramisequence}, one has
$$
\begin{aligned}
 2 E_{r, \mu}(u_{n,r})+ o_{n}(1) &= 2 E_{r, \mu}(u_{n,r})- \langle E_{r,\mu}^{\prime}(u_{n,r}), u_{n,r}\rangle\\
& =\int_{\Omega}f(u_{n,r})u_{n,r}\,dx-2\int_{\Omega}F(u_{n,r})\,dx+2f'_r\left(\frac{\norm{u_{n,r}}_2^2}{\mu}\right)\frac{\norm{u_{n,r}}^2_{2}}{\mu}-2f_r\left(\frac{\norm{u_{n,r}}_2^2}{\mu}\right).
\end{aligned}
$$
As $f(t)t-2F(t) \geq 0$ for all $t \in \R$, we obtain
$$
\begin{aligned}
 \limsup_{n \to +\infty}\left[f'_r\left(\frac{\norm{u_{n,r}}_2^2}{\mu}\right)\frac{\norm{u_{n,r}}^2_{2}}{\mu}-f_r\left(\frac{\norm{u_{n,r}}_2^2}{\mu}\right)\right] &\leq \lim_{n \to +\infty} E_{r, \mu}(u_{n,r}) \\
& =c_{r} \leq c_{\infty}=f_{r}^{\prime}(\xi_{r}) \xi_{r}-f_{r}(\xi_{r}).
\end{aligned}
$$
Employing the fact that $s \mapsto f'_{r}(s)$ and  $s \mapsto f_{r}^{\prime}(s) s-f_{r}(s)$ are strictly increasing on $[0,1)$, we deduce that
$$
\limsup _{n \to +\infty} \lambda_{n,r}=\frac{2}{\mu} \limsup _{n \to +\infty} f'_r\left(\frac{\norm{u_{n,r}}_2^2}{\mu}\right) \leq \frac{2}{\mu} f_{r}^{\prime}\left(\xi_{r}\right),
$$
then by \eqref{eqsupfr}, one has
$$
\limsup _{r \to +\infty} \lambda_{r} \leq \lim_{r \to +\infty}\frac{2}{\mu} f_{r}^{\prime}\left(\xi_{r}\right) = \frac{2c_{\infty}}{\mu}.
$$
This completes the proof.
\end{proof}

\begin{lemma}\label{lemunr}
    If $c_\infty:=\lim\limits_{r\to+\infty} c_r < + \infty$, then for $r>1$ sufficiently large, there exists $u_r \in U_\mu$ such that, up to a subsequence, $u_{n,r} \to u_r$ in $H^1_0(\Omega)$, as $n \to +\infty$. Moreover, $u_r$ satisfies $$
E_{r, \mu}(u_{r})=c_{r} \text{ and } E_{r, \mu}^{\prime}(u_{r})=0
$$
with
$$
 \frac{2}{\mu}f'_r\left(\frac{\norm{u_r}_2^2}{\mu}\right) = \lambda_{r}  \text{ and } \limsup_{r \to +\infty}\lambda_{r} \leq \frac{2c_\infty}{\mu}.
$$
\end{lemma}
\begin{proof}
    Let $r>1$ be sufficiently large. From Lemma \ref{lemlambdanr}, we may assume that $\lambda_{n,r} \to \lambda_{r}\leq \frac{2{c_\infty}}{\mu}$ as $n \to +\infty$. By Lemma \ref{lembounded}, $\{u_{n,r}\}_{n\geq 1}$ is bounded in $H^1_0(\Omega)$, and thus, up to a subsequence, $u_{n,r} \rightharpoonup u_{r}$ in $H^1_0(\Omega)$. Since $\{u_{n,r}\}_{n \geq 1} \subset U_\mu$, by Fatou's Lemma, we know that $\norm{u_{r}}_2^2\leq \mu$. We claim that $\norm{u_{r}}^2_2<\mu$. Argue by contradiction, suppose that $\norm{u_{r}}_2^2= \mu$. Then, $\displaystyle\lim_{n \to +\infty}\norm{u_{n,r}}_2^2= \mu$. Since $\{u_{n,r}\}_{n \geq 1}$ is bounded and $u_{n,r} \to u_r$ in $L^{2}(\Omega)$,  we obtain $E_{r,\mu}(u_{n,r}) \rightarrow -\infty$, which is absurd. Thus $u_r \in U_\mu$.

    Next we are going to prove that $u_{n,r} \to u_{r}$ in $H^1_0(\Omega)$. To this end, since the embedding $H^1_0(\Omega)\hookrightarrow L^p(\Omega)$ is compact for all $2\leq p<2^*$, it follows that $\langle E'(u_{r}), \cdot \rangle - \lambda_{r}(u_r,\cdot)_2=0$ and
\begin{equation*}
    \int_{\Omega} \left(f(u_{n,r})-f(u_{r})\right)(u_{n,r}-u_{r}) \, dx \to 0 \quad \text{ as }  n \to +\infty.
\end{equation*}
Since $\langle E_{r,\mu}'(u_{n,r}),u_{n,r}-u_r\rangle=o_n(1)$, the above information gives
\begin{equation*}
\begin{aligned}
    \norm{u_{n,r}-u_{r}}^{2}=&\langle E'(u_{n,r})-E'(u_r), u_{n,r}-u_r \rangle +\int_{\Omega} \left(f(u_{n,r})-f(u_{r})\right)(u_{n,r}-u_{r}) \, dx\\
    =&\langle E'_{r,\mu}(u_{n,r}), u_{n,r}-u_r \rangle -\langle E'(u_r), u_{n,r}-u_r \rangle+\lambda_{r}(u_r,u_{n,r}-u_r)_2 \\&+(\lambda_{n,r}  u_{n,r}-\lambda_{r} u_{r}, u_{n,r}-u_{r})_{2}+\int_{\Omega} \left(f(u_{n,r})-f(u_{r})\right)(u_{n,r}-u_{r}) \, dx\\
    =&\int_{\Omega} \left(f(u_{n,r})-f(u_{r})\right)(u_{n,r}-u_{r}) \, dx+(\lambda_{n,r}  u_{n,r}-\lambda_{r} u_{r}, u_{n,r}-u_{r})_{2}+o_{n}(1)\\
    =&(\lambda_{n,r}  u_{n,r}-\lambda_{r} u_{r}, u_{n,r}-u_{r})_{2}+o_{n}(1).
\end{aligned}
\end{equation*}
Consequently, by employing the compact embedding $H^1_0(\Omega)\hookrightarrow L^2(\Omega)$, we derive that
\begin{equation*}
   \begin{aligned}
   \norm{u_{n,r}-u_{r}}^{2}&=
(\lambda_{n,r}  u_{n,r}-\lambda_{r} u_{r}, u_{n,r}-u_{r})_{2}+o_{n}(1)\\
&=(\lambda_{n,r}-\lambda_{r})(u_{n,r}, u_{n,r}-u_{r})_{2}+\lambda_{r}((u_{n,r}-u_{r}), u_{n,r}-u_{r})_{2}+o_{n}(1)\\
&=\lambda_{r}\norm{u_{n,r}-u_{r}}_{2}^{2}+o_{n}(1)\\
&=o_{n}(1),
\end{aligned}
\end{equation*}
that is, $u_{n,r}\to u_{r}$ in $H^1_0(\Omega)$. Thus,
$$
E_{r,\mu}(u_{r}) = \lim\limits_{n \to +\infty}E_{r,\mu}(u_{n,r}) =c_r\,\,\text{and}\,\, E_{r,\mu}^\prime(u_{r}) = \lim\limits_{n \to +\infty}E_{r,\mu}^\prime(u_{n,r}) =0
$$
and
$$
\lambda_r = \lim_{n \to +\infty}\lambda_{n,r}=\lim_{n \to +\infty}\frac{2}{\mu}f'_r\left(\frac{\norm{u_{n,r}}_2^2}{\mu}\right) = \frac{2}{\mu}f'_r\left(\frac{\norm{u_r}_2^2}{\mu}\right).
$$
This completes the proof.
\end{proof}
Let $r_{n}>1$ be such that $r_{n} \nearrow +\infty$ (monotone sequences) and
$$
S_{1,0,0}(\mu)=\left\{u \in H_0^{1}(\Omega)\,:\,\int_{\Omega}u^2\,dx=\mu\right\}.
$$
\begin{lemma}\label{lemur}
     If  $c_\infty:=\lim\limits_{r\to+\infty} c_r <+\infty$, then, there exists $u \in H^1_0(\Omega)$ such that, up to a subsequence, $u_{r_n} \to u$ in $H^1_0(\Omega)$, as $n \to +\infty$. Moreover, $u$ satisfies $E(u)=c_\infty$, and
     \begin{enumerate}[label=\rm(\roman*)]
        \item either $u$ is a critical point of $E$ constrained on $S_{1,0,0}(\mu)$ with Lagrange multiplier $\lambda \in [0,\frac{2c_\infty}{\mu}]$.
\item  or $u$ is a critical point of $E$ constrained on $S_{1,0,0}(\nu)$ for some $0<\nu < \mu$ with Lagrange multiplier $\lambda =0$.
    \end{enumerate}
\end{lemma}
\begin{proof}
    Recall that $\{u_{r_n}\}_{n\geq 1}$ in Lemma \ref{lemunr} is a sequence such that
$$
E_{r_{n}, \mu}(u_{r_{n}})=c_{r} \text{ and } E_{r_{n}, \mu}^{\prime}(u_{r_{n}})=0
$$
with
$$
 \frac{2}{\mu}f'_{r_n}\left(\frac{\norm{u_{r_n}}_2^2}{\mu}\right) = \lambda_{r_n}  \text{ and } \limsup_{n \to +\infty}\lambda_{r_n} \leq \frac{2c_\infty}{\mu}.
$$
Repeating the arguments of Lemmas \ref{lembounded} and \ref{lemunr}, we may assume that $u_{r_n} \to u$ in $H^1_0(\Omega)$ and $\lambda_{r_n} \rightarrow \lambda$ as $n \to +\infty$. By \eqref{eqfr's}, we know that $f_{r_n}\left(\frac{\norm{u_{r_n}}_2^2}{\mu}\right) \to 0$ as $n \to +\infty$. Hence,
$E(u) = c_\infty$, $ \langle  E^{\prime}(u), \cdot\rangle-\lambda(u,\cdot)_2 = 0$, $\norm{u}^2_{2} \leq \mu$, $0 \leq \lambda \leq \frac{2c_\infty}{\mu}$.
Now, either $\norm{u}^2_{2}=\mu$ or $\norm{u}_{2}^2<\mu$. When the latter happens,
$$
0 \leq \lambda=\lim _{n \to +\infty}\frac{2}{\mu}f'_{r_n}\left(\frac{\norm{u_{r_n}}_2^2}{\mu}\right) \leq \limsup _{n \to +\infty} \frac{2}{\mu}f_{r_{n}}^{\prime}\left(\frac{\mu+\norm{u_{0}}_{2}^{2}}{2\mu}\right)=0,
$$
that is $\lambda=0$. This completes the proof.
\end{proof}

Next, we present a key lemma concerning a non-existence result, which will be used to rule out case (ii) in Lemma \ref{lemur}. For $M>0$ and $q >2$, define $$\mu^*:=\mu^*(K_2,K_p,\Omega,M,p,q):=
\begin{cases}
 \left(\frac{\lambda_{1,0,0}-K_2}{K_pC_{p,N}}\right)^{\frac{2}{p-2}}\lambda_{1,0,0}^{-\frac{N}{2}}  &\text{ if } 2<p\leq 2+\frac{4}{N},\\
    \left(\frac{\lambda_{1,0,0}-K_2}{K_pC_{p,N}\lambda_{1,0,0}}\right)^{\frac{2}{p-2}}\left(\frac{2qM}{q-2}\right)^{\frac{2}{p-2}-\frac{N}{2}} &\text{ if }  2+\frac{4}{N}<p<2^*.
\end{cases}$$

\begin{lemma}\label{lemnex}
The following two non-existence results hold:
    \begin{enumerate}[label=\rm(\roman*)]
        \item Assume that ($f_1$)-($f_3$) hold and that $\mu \in (0,1)$. Then, given $M>0$ independent of $\mu \in (0,1)$ and $r>0$, there exists a constant $\mu_0:=\mu_0(\Omega,f,M) \in (0,1)$ such that, for any $0<\nu<\mu_0$, there exists no $u \in H^1_0(\Omega)$ satisfying
             $$
             E'(u) = 0,\,\,E(u)\leq M,\,\,\text{and}\,\,\int_{\Omega}|u|^2\,dx=\nu.
             $$
        \item Assume that ($f_1$) and ($f_4$) hold. Then, for any $M>0$  independent of $\mu >0$ and $0<\nu<\mu^*$, there exists no $u \in H_0^1(\Omega)$ satisfying
             $$
             E'(u) = 0,\,\,E(u)\leq M\mu^*\,\,\text{and}\,\,\int_{\Omega}|u|^2 \,dx=\nu.
             $$
    \end{enumerate}
\end{lemma}

\begin{proof}
    (i) Suppose that ($f_1$)-($f_3$) hold and that $\mu \in (0,1)$. For any fixed $M>0$, we claim that, there exists $R>0$ independent of $\mu \in (0,1)$ and $r>1$, such that for all  $u \in H_0^1(\Omega)$ satisfying $E'(u) = 0$ and $E(u)\leq M$, we have $\norm{u}\leq R$. Indeed, if not, then there exists a sequence $\{u_n\}\subset H_0^1(\Omega)$ satisfying $E'(u_n) = 0$, $E(u_n)\leq M$
    and $\norm{u_n} \to +\infty$ as $n \to +\infty$. Repeating the arguments of Lemma \ref{lembounded}, we obtain a contradiction.

    Now, argue by contradiction, assume that for some $\nu>0$, there exists $u \in H_0^1(\Omega)$ satisfying $E'(u) = 0$, $E(u)\leq M$ and $\int_{\Omega}|u|^2\,dx=\nu$. By $E'(u) = 0$, ($f_1$) and \eqref{eqgn}, one has
    \begin{equation*}\label{lemnonex}
        \begin{aligned}
        \norm{u}^2 = \int_\Omega f(u)u\,dx &\leq K_2 \norm{u}_2^2+K_{p} C_{p,N}\norm{u}^{\beta_{p}{p}}\norm{u}_2^{(1-\beta_{p}){p}} \\
        &\leq \frac{K_2}{\lambda_{1,0,0}} \norm{u}^2+K_{p} C_{p,N}\norm{u}^{\beta_{p}{p}}\norm{u}_2^{(1-\beta_{p}){p}}.
        \end{aligned}
    \end{equation*}
    Hence,
    \begin{equation*}
    \frac{\lambda_{1,0,0}-K_2}{\lambda_{1,0,0}}\leq K_{p} C_{p,N}\norm{u}^{\beta_{p}{p}-2}\nu^\frac{(1-\beta_{p}){p}}{2}.
    \end{equation*}
    If $p \leq 2+\frac{4}{N}$, then $\beta_pp\leq 2$ and
    $$
    \frac{\lambda_{1,0,0}-K_2}{\lambda_{1,0,0}}\leq K_{p} C_{p,N}\lambda_{1,0,0}^\frac{{\beta_{p}{p}-2}}{2}\nu^\frac{{\beta_{p}{p}-2}}{2} \nu^\frac{(1-\beta_{p}){p}}{2}=K_{p} C_{p,N} \lambda_{1,0,0}^\frac{{\beta_{p}{p}-2}}{2}\nu^\frac{p-2}{2}.
    $$
     If $p > 2+\frac{4}{N}$, then $\beta_pp > 2$ and
    $$
    \frac{\lambda_{1,0,0}-K_2}{\lambda_{1,0,0}}\leq K_{p} C_{p,N}R^{\beta_{p}{p}-2} \nu^\frac{(1-\beta_{p}){p}}{2}.
    $$
    Therefore, there exists $\mu_0>0$ such that, for any $0<\nu<\mu_0$, there exists no $u \in H_0^1(\Omega)$ satisfying
     $$
     E'(u) = 0,\,\,E(u)\leq M\,\,\text{and}\,\,\int_{\Omega}|u|^2 \,dx=\nu.
     $$

    (ii) Suppose that ($f_1$) and ($f_4$) hold. Given $M>0$, argue by contradiction, assume that there exists $u \in H_0^1(\Omega)$ satisfying $E'(u) = 0$, $E(u)\leq M\mu^*$ and $\int_{\Omega}|u|^2\,dx=\nu$ for some $0<\nu<\mu^*$. Then, by ($f_4$), we have
    $$
    \norm{u}^2 \leq \frac{2qE(u)-2\langle E'(u),u\rangle}{q-2}\leq \frac{2qM\mu^*}{q-2}.
    $$
    Repeating the previous argument, we derive that, if $2<p\leq2+\frac{4}{N}$, then
    $$
    \frac{\lambda_{1,0,0}-K_2}{\lambda_{1,0,0}}\leq K_{p} C_{p,N} \lambda_{1,0,0}^\frac{{\beta_{p}{p}-2}}{2}\nu^\frac{p-2}{2},
    $$
    and, if $2+\frac{4}{N}<p<2^*$, then
    $$
    \begin{aligned}
        \frac{\lambda_{1,0,0}-K_2}{\lambda_{1,0,0}}&\leq K_{p} C_{p,N}\norm{u}^{\beta_{p}{p}-2}\nu^\frac{(1-\beta_{p}){p}}{2}\\
        &\leq K_{p} C_{p,N}\left(\frac{2qM\mu^*}{q-2}\right)^\frac{{\beta_{p}{p}-2}}{2}{\nu}^\frac{(1-\beta_p) p}{2}\\
        &<K_{p} C_{p,N}\left(\frac{2qM}{q-2}\right)^\frac{{\beta_{p}{p}-2}}{2}{\mu^*}^\frac{p-2}{2},
    \end{aligned}
    $$
    which leads to a contradiction, since $\frac{2-\beta_pp}{p-2}=\frac{2}{p-2}-\frac{N}{2}$ and
 $$\nu < \mu^*=
\begin{cases}
    \left(\frac{\lambda_{1,0,0}-K_2}{K_pC_{p,N}}\right)^{\frac{2}{p-2}}\lambda_{1,0,0}^{-\frac{N}{2}}  &\text{ if } 2<p\leq 2+\frac{4}{N},\\
    \left(\frac{\lambda_{1,0,0}-K_2}{K_pC_{p,N}\lambda_{1,0,0}}\right)^{\frac{2}{p-2}}\left(\frac{2qM}{q-2}\right)^{\frac{2}{p-2}-\frac{N}{2}} &\text{ if }  2+\frac{4}{N}<p<2^*.
\end{cases}$$
\end{proof}
\subsection{Proof of Theorems \ref{th2} and \ref{th3}}\label{secp23}
The following lemma shows that the functional $E_{r, \mu}$ possesses the mountain pass geometry, thus we can apply Theorem \ref{thmpopen} in appendix.
\begin{lemma}\label{lemmpg}
The functional $E_{r, \mu}$ possesses the following properties:
  \begin{enumerate}[label=\rm(\roman*)]
\item There exist $\alpha, \rho>0$ such that $E_{r, \mu}(u) \geq \alpha$ for any $u \in  U_{\mu}$ with $\|u\|=\rho$;
\item There exists  $e \in U_{\mu}$ with $\|e\|>\rho$ such that $E_{r, \mu}(e)<0$.
\end{enumerate}
\end{lemma}
\begin{proof}
    (i) For any $u \in U_{\mu}$ with $\|u\|=\rho$, where $\rho>0$ is small enough, we have
$$
\norm{u}_{2}^2 \leq \frac{1}{\lambda_{1,0,0}}\|u\|^{2}=\frac{\rho^{2}}{\lambda_{1,0,0}} <\mu.
$$
Thus, by the definition of $E_{r, \mu}(u)$, the monotonicity of $f_{r}(s)$ with respect to $s$, ($f_1$) and \eqref{eqgn}, we obtain
$$
\begin{aligned}
E_{r, \mu}(u) & =\frac{1}{2}\|u\|^{2}-\Psi(u)-f_{r}\left(\frac{\norm{u}^2_2}{\mu}\right) \\
& \geq \frac{1}{2}\|u\|^{2}-\Psi(u)-f_{r}\left(\frac{\rho^{2}}{\mu\lambda_{1,0,0}}\right) \\
& \geq \rho^{2}\left(\frac{1}{2}-\frac{K_2}{2\lambda_{1,0,0}}-C \rho^{p-2}-\frac{(\mu\lambda_{1,0,0})^{-r} \rho^{2 r-2}}{1-(\mu\lambda_{1,0,0})^{-1}\rho^{2}}\right).
\end{aligned}
$$
Since $K_2 <\lambda_{1,0,0}$, $2<p<2^*$ and $r>1$, there exist $\alpha, \rho>0$ small enough such that
$$
E_{r, \mu}(u) \geq \alpha>0 \quad \text{ for any}\,\,  u \in U_{\mu} \text { with }\|u\|=\rho.
$$
(ii) Choosing $u_{0} \in H^1_0(\Omega)$ with $\norm{u_0}^2_2=\mu$, it is easy to verify that
$$
\lim _{t \rightarrow 1^{-}} E_{r, \mu}(tu_0)=-\infty.
$$
Therefore, there exists $t_{0}<1$ such that $\norm{t_{0} u_{0}}\geq t_0\sqrt{\mu\lambda_{1,0,0}}>\rho$ and $ E_{r, \mu}(t_0u_0)<0$, then (ii) holds.
\end{proof}
By Theorem \ref{thmpopen},  the minimax value
$$
c_{r}:=\inf _{\gamma \in \Gamma_{r, \mu}} \max _{t \in[0,1]} E_{r, \mu}(\gamma(t))>0,
$$
is well defined, where
$$
\Gamma_{r,\mu} := \{\gamma \in C([0,1], U_{\mu}): \gamma(0)=0,\,  E_{r, \mu}(\gamma(1))<0\}.
$$
 If $r_{1} \leq r_{2}$, we have $c_{r_{1}} \leq c_{r_{2}}$ because $E_{r_1, \mu} \leq E_{r_2, \mu}$ on $U_{\mu}$.

For any $r>1$, the mountain pass geometry allows us to find a Cerami sequence $\{u_{n,r}\}_{n\geq 1}$  satisfying
\begin{equation*}
 E_{r, \mu}(u_{n,r}) \rightarrow c_{r}, \quad (1+\norm{u_{n,r}})\norm{E_{r, \mu}^{\prime}(u_{n,r})} \rightarrow 0.
\end{equation*}
Define $c_\infty=c_{\infty}(\mu):=\sup\limits_{r>1} c_{r}=\lim\limits_{r\to + \infty} c_{r}$. For all $u \in H^1_0(\Omega) \backslash\{0\}$ with $\norm{u}_2^2 = \mu$, we have
\begin{equation}\label{eqcssup}
 c_{\infty} \leq \sup _{0 \leq t<1} E(t u).
\end{equation}
Now, we provide an important upper bound estimate for $c_\infty$.
\begin{lemma}\label{lemcin}
 $c_{\infty}  \leq \frac{\mu\lambda_{1,0,0}}{2} < +\infty$.
\end{lemma}
\begin{proof}
Let $\varphi_1$ be the eigenvector corresponding to the first eigenvalue $\lambda_{1,0,0}$ with $\norm{\varphi_1}^2_2=\mu$. Then, since $\Psi(t\varphi_1) \geq 0$ for all $t \geq  0$, we conclude
$$
\sup _{0 \leq t<1}E(t\varphi_1) = \sup _{0 \leq t<1}\left[\frac{t^2\mu\lambda_{1,0,0}}{2} - \Psi(t\varphi_1)\right]\leq \frac{\mu\lambda_{1,0,0}}{2}.
$$
\end{proof}
It is the position to provide the proof of Theorem \ref{th2}-\ref{th3}.
\begin{proof}[Proof of Theorem \ref{th2}-\ref{th3}]
By Lemma \ref{lemur} and Lemma \ref{lemcin}, we know that there exists $u \in H^1_0(\Omega)$ satisfying $E(u)=c_\infty\leq \frac{\mu\lambda_{1,0,0}}{2}$, and one of the following two cases must hold:
     \begin{enumerate}[label=\rm(\roman*)]
\item either $u$ is a critical point of $E$ constrained on $S_{1,0,0}(\mu)$ with a Lagrange multiplier $\lambda \in [0,\frac{2c_\infty}{\mu}]$.
\item  or $u$ is a critical point of $E$ constrained on $S_{1,0,0}(\nu)$ for some $0<\nu < \mu$ with a Lagrange multiplier $\lambda =0$.
    \end{enumerate}
We first assume that ($f_1$)-($f_3$) hold and that $\mu \in (0,1)$. Applying Lemma \ref{lemnonex} with $M=\frac{\lambda_{1,0,0}}{2}$, for $0<\mu <\mu_0$, we conclude that the second case above can not occur.

Now, assume that ($f_1$) and ($f_4$) hold. As in Lemma \ref{lemnonex}, setting
$$
\mu^*:=\mu^*(K_2,K_p,\Omega,\frac{\lambda_{1,0,0}}{2},p,q)=
\begin{cases}
 \left(\frac{\lambda_{1,0,0}-K_2}{K_pC_{p,N}}\right)^{\frac{2}{p-2}}\lambda_{1,0,0}^{-\frac{N}{2}}  &\text{ if } 2<p\leq 2+\frac{4}{N},\\
    \left(\frac{\lambda_{1,0,0}-K_2}{K_pC_{p,N}}\right)^{\frac{2}{p-2}}\left(\frac{q}{q-2}\right)^{{\frac{2}{p-2}}-\frac{N}{2}}\lambda_{1,0,0}^{-\frac{N}{2}} &\text{ if }  2+\frac{4}{N}<p<2^*.
\end{cases}
$$
By Lemma \ref{lemnonex}, we conclude that, for any $0<\mu\leq\mu^*$,
the second case above can not occur.
\end{proof}
\begin{proof}[Proof of Corollary \ref{co1}]
(i) follows immediately from Theorem \ref{th3}.

(ii) Now, assume that ($f_4$) hold, and suppose that $|f(t)|\leq a\left(\abs{t}^{p'-1} +\abs{t}^{p-1}\right)$  for some constant $a>0$, where $2+\frac{4}{N}<p'<p<2^*$. Since
$$
a\abs{t}^{p'-1}\leq \begin{cases}
    \frac{\lambda_{1,0,0}}{2}\abs{t}\quad &\text{if }\abs{t}\leq \left(\frac{\lambda_{1,0,0}}{2a}\right)^\frac{1}{p'-2},\\
    a\left(\frac{\lambda_{1,0,0}}{2a}\right)^\frac{p'-p}{p'-2}\abs{t}^{p-1} &\text{if }\abs{t}> \left(\frac{\lambda_{1,0,0}}{2a}\right)^\frac{1}{p'-2},
\end{cases}
$$
we have
    $$
        |f(t)|\leq \frac{\lambda_{1,0,0}}{2}\abs{t} +a\left(1+\left(\frac{\lambda_{1,0,0}}{2a}\right)^\frac{p'-p}{p'-2}\right)\abs{t}^{p-1}.
    $$
   Setting $K_2 =\frac{\lambda_{1,0,0}}{2}$ and $K_p=a+a\left(\frac{\lambda_{1,0,0}}{2a}\right)^\frac{p'-p}{p'-2}$. As in Theorem \ref{th3}, setting
$$
\begin{aligned}
\mu^*:=&\left(\frac{\lambda_{1,0,0}-K_2}{K_pC_{p,N}}\right)^{\frac{2}{p-2}}\left(\frac{q}{q-2}\right)^{{\frac{2}{p-2}}-\frac{N}{2}}\lambda_{1,0,0}^{-\frac{N}{2}}\\
=&\frac{\lambda_{1,0,0}^{{\frac{2}{p-2}}-\frac{N}{2}}}{\left(aC_{p,N}+aC_{p,N}\left(\frac{\lambda_{1,0,0}}{2a}\right)^\frac{p'-p}{p'-2}\right)^\frac{2}{p-2}}\left(\frac{q}{q-2}\right)^{\frac{2}{p-2}-\frac{N}{2}}\left(\frac{1}{2}\right)^{\frac{2}{p-2}}.
\end{aligned}
$$  By Theorem \ref{th3}, we know that, if  $0<\mu \leq \mu^*$ holds,
   then $(P)^{\mu}_{1,0,0}$ admits a solution $u \in H^1_0(\Omega)$ with $\lambda \in [0, \lambda_{1,0,0}]$. Moreover, we compute that
   $$
   p'-p<0,\,\,\,\,\,\frac{2}{p-2}-\frac{N}{2}-\frac{2(p'-p)}{(p'-2)(p-2)}=\frac{2}{p'-2}-\frac{N}{2}<0,
   $$
   hence
   $$
   \frac{\lambda_{1,0,0}^{{\frac{2}{p-2}}-\frac{N}{2}}}{\left(aC_{p,N}+aC_{p,N}\left(\frac{\lambda_{1,0,0}}{2a}\right)^\frac{p'-p}{p'-2}\right)^\frac{2}{p-2}}\to +\infty \quad \text{as }\lambda_{1,0,0} \to 0^+.
   $$
   Thus, there exists a constant $\lambda^*>0$ depending on $a$, $p$, $p'$, $q$, $N$ and $\mu$ such that, if $0<\lambda_{1,0,0}<\lambda^*$, then $(P)^{\mu}_{1,0,0}$ admits a solution $u \in H^1_0(\Omega)$ with $\lambda \in [0, \lambda_{1,0,0}]$.
\end{proof}
\subsection{Proof of Theorem \ref{th4}}\label{secp4}
In this subsection, we will consider the multiplicity of normalized solutions to problem $(P)^{\mu}_{1,0,0}$ and prove Theorem \ref{th4}. Throughout, we assume that either($f_1$)-($f_3$) hold or ($f_1$) and ($f_4$) hold. Moreover, we also assume that $f(t)=-f(-t)$.

By \cite[Theorem 1 in Section 6.5]{evans}, $(-\Delta,H^{1}_0(\Omega))$ has an increasing sequence of infinity distinct eigenvalues $0<\lambda_1=\lambda_{1,0,0}<\lambda_{2}<\lambda_{3}<\cdots$. For any $2 \leq j <+\infty$, let $\varphi_j$ be an eigenfunction corresponding to the eigenvalue $\lambda_{j}$.
For $j \geq 2$ and $r>1$, define
$$
\begin{aligned}
    &Y_j:=\operatorname{span}\{u \in H^1_0(\Omega):-\Delta u=\lambda u \text{ for some }\lambda \leq \lambda_{j}\},\\
    & Z_j:=Y_j^\perp\cup\operatorname{span}\{\varphi_j\},
\end{aligned}
$$
and define
$$
B_{r,j}:=\{u\in Y_j  : \norm{u}\leq \rho_{r,j}\}, \quad N_{r,j}:=\{u\in Z_j : \norm{u}= \xi_{r,j}\},
$$
where $\rho_{r,j}>\xi_{r,j}>0$.  By the Fountain theorem \cite[Theorem 3.5]{W}, for the functional $J_{r,\mu}$ defined in Section \ref{subsecpf}, see \eqref{Jrmu}, we have the following result.
\begin{lemma}
    For $r>1$ and $j \geq 2$, define
    \begin{equation*}\begin{aligned}
&c_{r,j} := \inf_{\gamma \in \Gamma_{r,j}} \max_{u \in B_{r,j}} J_{r,\mu}(\gamma(u)),\\
&\Gamma_{r,j} := \{\gamma \in C(B_{r,j}, H^1_0(\Omega)) : -\gamma(u)=\gamma(-u) \text{ for all } u \in B_{r,j}\text{ and } \gamma\left.\right|_{\partial B_{r,j}} = \operatorname{id}\}.
    \end{aligned}
\end{equation*}
If
$$
b_{r,j} := \inf_{\begin{subarray}{c} u \in Z_j \\ \norm{u}= \xi_{r,j}\end{subarray}} J_{r,\mu}(u) > 0 > a_{r,j} := \max_{\begin{subarray}{c} u \in Y_j \\ \norm{u} = \rho_{r,j} \end{subarray}} J_{r,\mu}(u),
$$
 then \( c_{r,j} \geq b_{r,j} >0 \), and, for $E_{r, \mu}$, there exists a Palais-Smale sequence $\{u_{n,r,j}\}_{n\geq 1}$ satisfying
\begin{equation*}
 E_{r, \mu}(u_{n,r,j}) \rightarrow c_{r,j} \quad \text{and} \quad E_{r, \mu}^{\prime}(u_{n,r,j}) \rightarrow 0 \quad\text{as }n \to +\infty.
\end{equation*}
\end{lemma}
In view of \cite{Bar,Ce}, the existence of Palais-Smale sequence $\{u_{n,r,j}\}_{n\geq1}$ in Lemma \ref{lemfountain} can be replaced by the existence of Cerami sequence $\{u_{n,r,j}\}_{n\geq1}$ satisfying
\begin{equation*}
 E_{r, \mu}(u_{n,r,j}) \rightarrow c_{r,j}\quad \text{and} \quad (1+\norm{u_{n,r,j}})\norm{E_{r, \mu}^{\prime}(u_{n,r,j})} \rightarrow 0 \quad\text{as }n \to +\infty.
\end{equation*}

To apply Lemma \ref{lemur},  we first establish some basic frameworks as in Section \ref{secp23}.

For $r>1$ , $j \geq 2$, set
    $$\rho_{r,j}= \sqrt{\mu\lambda_{j} }.$$
\begin{lemma}\label{lemfountain}
    For $r>1$ and $j \geq 2$,
    $$c_{r,j} \leq \frac{\mu \lambda_{j}}{2} \quad \text{ and } \quad a_{r,j}=-1.$$
Moreover, for any $j \geq 2$ and $\varepsilon>0$, there exist $\xi_{\varepsilon,j}>0$ and $r_{\varepsilon,j}>1$, such that, for all $r\geq r_{\varepsilon,j}$ and $\xi_{r,j}=\xi_{\varepsilon,j}$,
$$
\widetilde{b_{r,j}}:=\inf_{\begin{subarray}{c} u \in Z_j \\ \norm{u}= \xi_{r,j}\end{subarray}} E_{r,\mu}(u) = \inf_{\begin{subarray}{c} u \in Z_j \\ \norm{u}= \xi_{\varepsilon,j}\end{subarray}} E_{r,\mu}(u)  \geq
\frac{\mu(\lambda_{j}-K_2)}{2}-\frac{K_pC_{p,N}\mu^\frac{p}{2}\lambda_{j}^\frac{\beta_pp}{2}}{p}-\varepsilon
$$
\end{lemma}
\begin{proof}
For all $u \in Y_j$, we have
$$
\norm{u}^2\leq \lambda_{j} \norm{u}^2_2.
$$
Then,
$$
\sup_{u \in Y_j \cap U_\mu}E_{r,\mu}(\operatorname{id}(u)) = \sup_{u \in Y_j \cap U_\mu}E_{r,\mu}(u) \leq \frac{\mu\lambda_{j}}{2}.
$$
Thus,
$$
c_{r,j} \leq \sup_{u \in Y_j \cap U_\mu}J_{r,\mu}(\operatorname{id}(u)) \leq \max\left\{0, \sup_{u \in Y_j \cap U_\mu}E_{r,\mu}(\operatorname{id}(u))\right\} \leq \frac{\mu \lambda_{j}}{2}.
$$
Moreover, it follows from $\rho_{r,j}= \sqrt{\mu\lambda_{j} }$ that
    $$
     \norm{u}^2_2 \geq \mu, \quad \mbox{for} \quad \|u\|=\rho_{r,j}.
    $$
Hence, $a_{r,j}=-1$.

For some $k>1$ and $k\approx +\infty$, set
$$
\xi_{\varepsilon,j}= \sqrt{\frac{k-1}{k}\mu\lambda_{j} }.
$$
For all $u \in Z_j$ with $\norm{u}=\xi_{\varepsilon,j}$, one has
    $$
     \norm{u}^2_2 \leq \frac{\norm{u}^2}{\lambda_{j}} = \frac{\mu(k-1)}{k},
    $$
and thus, by \eqref{eqgn} and ($f_1$),
$$
E_{r,\mu}(u)\geq \frac{k-1}{k}\frac{\mu\lambda_{j}}{2}-\frac{(k-1)}{k}\frac{K_2\mu}{2}-\left(\frac{k-1}{k}\right)^\frac{p}{2}\frac{K_pC_{p,N}\mu^\frac{p}{2}\lambda_{j}^\frac{\beta_pp}{2}}{p}- k\left(\frac{k-1}{k}\right)^r.
$$
Thus, for any $\varepsilon>0$, there exist $\xi_{\varepsilon,j}>0$ and $r_{\varepsilon,j}>1$, such that, for all $r\geq r_{\varepsilon,j}$ and $\xi_{r,j}=\xi_{\varepsilon,j}$,
$$
\widetilde{b_{r,j}}\geq \frac{\mu(\lambda_{j}-K_2)}{2}-\frac{K_pC_{p,N}\mu^\frac{p}{2}\lambda_{j}^\frac{\beta_pp}{2}}{p}-\varepsilon.
$$
\end{proof}
Now, we are ready to prove Theorem \ref{th4}.
\begin{proof}[Proof of Theorem \ref{th4}]
For any $j \geq 2$, if $1<r_{1} \leq r_{2}$ and $c_{r_{1},j}>0$, then, for all $\gamma \in \Gamma_{r_1,j}$, we have $\displaystyle\max_{u \in B_{r_1,j}} J_{r_1,\mu}(\gamma(u))>0$, and thus $$\max_{u \in B_{r_1,j}} J_{r_1, \mu}(\gamma(u))=\max_{u \in B_{r_1,j}} E_{r_1, \mu}(\gamma(u)).$$ Since $\rho_{r,j}, B_{r,j}$ and $\Gamma_{r,j}$  are independent of $r>1$ and $E_{r_1, \mu} \leq E_{r_2, \mu}$ on $H^1_0(\Omega)$,
     \begin{equation}\label{eqcr1r2}
    \begin{aligned}
    c_{r_{1},j}= \inf_{\gamma \in \Gamma_{r_1,j}} \max_{u \in B_{r_1,j}} J_{r_1,\mu}(\gamma(u))= \inf_{\gamma \in \Gamma_{r_1,j}} \max_{u \in B_{r_1,j}} E_{r_1, \mu}(\gamma(u))\leq\inf_{\gamma \in \Gamma_{r_2,j}} \max_{u \in B_{r_2,j}} E_{r_2, \mu}(\gamma(u)).
     \end{aligned}
     \end{equation}
     Hence, for all $\gamma \in \Gamma_{r_2,j}$, we have $\displaystyle\max_{u \in B_{r_2,j}} E_{r_2,\mu}(\gamma(u))>0$. Then, by \eqref{eqcr1r2},
     $$
     c_{r_{1},j}\leq\inf_{\gamma \in \Gamma_{r_2,j}} \max_{u \in B_{r_2,j}} E_{r_2, \mu}(\gamma(u))=\inf_{\gamma \in \Gamma_{r_2,j}} \max_{u \in B_{r_2,j}} J_{r_2,\mu}(\gamma(u))=c_{r_{2},j}.
     $$
     For $j \geq 2$, if $c_{r_{1},j}>0$ for some $r_1 >1$, then define $c_{\infty,j}:=\lim_{r \to +\infty}c_{r,j}$ and denote $c_{\infty,1}:=c_\infty$, which is defined in Lemma \ref{lemcin}.

By the definition of $J_{r,\mu}(u)$, we know that, if $\widetilde{b_{r,j}}>0$, then $b_{r,j}=\widetilde{b_{r,j}}>0$. Let $k_1=1$. Fix $m \geq 2$. Since $\lambda_{j} \to +\infty$ as $j \to +\infty$, we can choose $\{k_i\}_{i=2}^m$ such that
$$\lambda_{1,0,0} < \lambda_{k_2}-K_2 \leq \lambda_{k_2}<\cdots<\lambda_{k_{m-1}}-K_2\leq \lambda_{k_{m-1}} < \lambda_{k_{m}}-K_2\leq \lambda_{k_{m}}.$$ Then, there exists $\widetilde{\mu_{k_m}}$ such that, for all $0<\mu<\widetilde{\mu_{k_m}}$,
$$
\frac{\mu \lambda_{1,0,0}}{2}< \frac{\mu(\lambda_{k_2}-K_2)}{2}-\frac{K_pC_{p,N}\mu^\frac{p}{2}\lambda_{k_2}^\frac{\beta_pp}{2}}{p} \leq \frac{\mu\lambda_{k_2}}{2}<\cdots <\frac{\mu(\lambda_{k_{m}}-K_2)}{2}-\frac{K_pC_{p,N}\mu^\frac{p}{2}\lambda_{k_{m}}^\frac{\beta_pp}{2}}{p} \leq \frac{\mu\lambda_{k_{m}}}{2}.
$$
Fix $0<\mu<\widetilde{\mu_{k_m}}$. By Lemma \ref{lemfountain}, for any $2 \leq i \leq m$, there exist $\xi_{k_i}>0$ and $r_{k_i}>1$, such that, for all $r\geq r_{k_i}$ and $\xi_{r,k_i}=\xi_{k_i}$,
$$
\frac{\mu\lambda_{k_{i}}}{2}\geq c_{r,k_i}\geq b_{r,k_i}=\widetilde{b_{r,k_i}}=\inf_{\begin{subarray}{c} u \in Z_{k_i} \\ \norm{u}= \xi_{r,k_i}\end{subarray}} E_{r,\mu}(u)  > \frac{\mu\lambda_{k_{i-1}}}{2}.
$$
Let $\displaystyle r_0 = \max_{2\leq i\leq m}\{r_{k_i}\}$ and $\xi_{r,k_i}=\xi_{k_i}$. Thus, for $r \geq r_0$, we conclude
$$
    0 < c_{\infty,1} \leq \frac{\mu\lambda_{1}}{2} < b_{r,k_2} \leq c_{\infty,k_2} \leq \frac{\mu\lambda_{k_2}}{2} <\cdots <
    b_{r,k_m}\leq c_{\infty,k_m} \leq \frac{\mu\lambda_{k_m}}{2}.
$$

    Since $\lambda_{k_m} < +\infty$, by Lemma \ref{lemur} and \ref{lemcin}, we conclude that, for any $2 \leq i \leq m$, there exists $u_{k_i} \in H^1_0(\Omega)$ satisfying $E(u_{k_i})=c_{\infty,k_i}$ and $c_{\infty,k_{i-1}}<c_{\infty,k_i}\leq \frac{\mu\lambda_{k_i}}{2}$, and one of the following two cases must hold:
     \begin{enumerate}[label=\rm(\roman*)]
        \item either $u_{k_i}$ is a critical point of $E$ constrained on $S_{1,0,0}(\mu)$ with a Lagrange multiplier $\omega_{k_i} \in [0,\lambda_{k_i}]$.
        \item  or $u_{k_i}$ is a critical point of $E$ constrained on $S_{1,0,0}(\nu)$ for some $0<\nu < \mu$ with a Lagrange multiplier $\omega_{k_i} =0$.
    \end{enumerate}
First, assume that ($f_1$)-($f_3$) hold and that $\mu \in (0,1)$. Applying Lemma \ref{lemnonex} with $M=\frac{\lambda_{k_m}}{2}$, for \linebreak $0<\mu <\mu^*_{m} :=\min\{\mu_0,\widetilde{\mu_{k_m}}\}$, we conclude that, for any $1\leq i\leq m$, the second case can not occur, and thus, $E$ has at least $m$ critical points $u_{k_1}, u_{k_2},\cdots,u_{k_m}$ constrained on $S_{1,0,0}(\mu)$.

    Now, we assume assume that ($f_1$) and ($f_4$) hold. Applying Lemma \ref{lemnonex}, by fixing
    $$\mu^*(K_2,K_p,\Omega,\frac{\lambda_{k_m}}{2},p,q)=
\begin{cases}
 \left(\frac{\lambda_{1}-K_2}{K_pC_{p,N}}\right)^{\frac{2}{p-2}}\lambda_{1}^{-\frac{N}{2}}  &\text{ if } 2<p\leq 2+\frac{4}{N},\\
    \left(\frac{\lambda_{1}-K_2}{K_pC_{p,N}\lambda_{1}}\right)^{\frac{2}{p-2}}\left(\frac{q\lambda_{k_m}}{q-2}\right)^{{\frac{2}{p-2}}-\frac{N}{2}} &\text{ if }  2+\frac{4}{N}<p<2^*,
\end{cases}
    $$
and $0<\mu\leq \mu^*_{m} := \min\{\mu^*(K_2,K_p,\Omega,\frac{\lambda_{k_m}}{2},p,q), \widetilde{\mu_{k_m}}\}$, we derive that, for any $1\leq i\leq m$, the second case can not occur. Thus, $E$ has at least $m$ critical points $u_{k_1}, u_{k_2},\cdots,u_{k_m}$ constrained on $S_{1,0,0}(\mu)$.

\end{proof}
\section{The problem \texorpdfstring{$(P)^\mu_{1,1,1}$}{Lg}}\label{secp111}
In this section,  our goal is to study the following problem
$$
\left\{
\begin{array}{ll}
	-\Delta u= \lambda u + f(u)\quad & \text{in } \Omega, \\
	\frac{\partial u}{\partial \eta} = -u+g(u) & \text{on } \partial \Omega, \\
	\int_{\Omega}u^2\,dx = \mu,
\end{array}
\right.
\leqno{{(P)}^{\mu}_{1,1,1}}
$$
under the assumptions that $f,g$ satisfy the condition stated in the introduction, namely ($f_1'$), ($f_4$), ($g_1$) and ($g_2$).

Note that, since conditions ($f_4$) and ($g_2$) hold, there exists constants $q'>2$ and $q''>2$ such that
\begin{equation*}
       0 \leq q'F(t)\leq f(t)t \quad \text{and} \quad 0 \leq q''G(t)\leq g(t)t.
\end{equation*}
Therefore, set $q=\min\{ q',q''\}$, we have
\begin{equation}\label{eqarq}
    0 \leq qF(t)\leq f(t)t \quad \text{and} \quad 0 \leq qG(t)\leq g(t)t.
\end{equation}
As a consequence, without loss of generality, we assume that \eqref{eqarq} holds in place of ($f_4$) and ($g_2$).

By \cite[Theorem 1.2 and Theorem 1.9 in Section 1.1]{Ne}, we can define an equivalent norm on $H^1(\Omega)$ by
$$
\|u\|:=\left(\int_{\Omega}|\nabla u|^2\,dx+ \int_{\partial \Omega}|u|^2\,d\sigma\right)^{\frac{1}{2}}.
$$
The energy functional $E: H^1(\Omega) \to \mathbb{R}$ associated with  problem $(P)^\mu_{1,1,1}$ is given by
$$
E(u) := \frac{1}{2}\norm{u}^{2}-\Psi(u),
$$
where $$
\Psi(u)=\int_\Omega  F(u)\,dx + \int_{\partial \Omega} G(u)\,d\sigma.
$$

It follows from the Rellich-Kondrachov theorem that $H^1(\Omega)$ is compactly embedded in $L^p(\Omega)$ for all $2\leq p<2^*$.  Moreover, by the Sobolev embedding theorem, for any $p \in [2, 2^*)$, we have
\begin{equation}\label{eqgnh1}
    \|u\|_p^p \leq C_{p,\Omega} \| u\|^{p}, \,\,\text{for any}\,\, u \in  H^1(\Omega),
\end{equation}
where $ C_{p,\Omega} > 0$ depends on $p$ and $\Omega$.

By \cite[Theorem 6.2 in Section 2.6]{Ne}, we know that $H^1(\Omega)$ is compactly embedded in $L^l(\partial \Omega)$ for all $2\leq l<2^{\hexstar}$ (for $N=2$ and $2\leq l<\infty$,  we use the facts that, for some $\varepsilon>0$ sufficiently small, $H^1(\Omega)$ is continuously embedded in $W^{1,2-\varepsilon}(\Omega)$ and $W^{1,2-\varepsilon}(\Omega)$ is compactly embedded in $L^l(\partial \Omega)$) and
\begin{equation}
  \label{eqsi}
    \|u\|_{L^l(\partial\Omega)}^l \leq C_{l,\Omega}' \|u\|^{l} \quad \forall u \in  H^1(\Omega),
\end{equation}
where $C_{l,\Omega}'$ depends on $l$ and $\Omega$.

We study the existence and multiplicity of critical points of the functional $E$ constrained on the $L^2$-sphere$$S_{1,1,1}(\mu)= \left\{ u \in H^1(\Omega):\int_\Omega u^2\,dx = \mu \right\},$$where $\mu > 0$.

To this end, we introduce the following perturbation functional
$$
E_{r, \mu}(u) := \frac{1}{2}\norm{u}^{2}-\Psi(u)-H_{r,\mu}(u),\; u \in U_{\mu},
$$
where $U_{\mu} := \{u \in H^1(\Omega): \norm{u}_2^2<\mu\}$, and the penalization term $H_{r, \mu}(u)$ is defined as
$$
H_{r, \mu}(u) := f_{r}\left(\frac{\norm{u}_2^2}{\mu}\right) \text { with } f_{r}(s) := \frac{s^{r}}{1-s} \; \quad \forall s \in [0,1),
$$
where $r>1$ is a parameter that will be chosen large enough.

The proofs of Theorem \ref{th2.3} and \ref{th4.3} are almost the same as the proof of Theorem \ref{th2} and \ref{th4}, respectively, except for Lemma \ref{lemnonnomass}. Therefore, we only present a sketch of the proofs of Theorem \ref{th2.3} and \ref{th4.3} here.

Firstly, we provide a sketch of the proof of Theorem \ref{th2.3}. By repeating the argument used in Lemma \ref{lemmpg}, we conclude that, for any $u \in U_{\mu}$ with $\|u\|=\rho$, where $\rho>0$ is small enough, we have
$$
\norm{u}_{2}^2 \leq \frac{1}{\hat{\lambda}_{1,1,1}}\|u\|^{2}=\frac{\rho^{2}}{\hat{\lambda}_{1,1,1}} <\mu.
$$
Moreover, by \eqref{eqgnh1} and \eqref{eqsi}, we obtain
$$
\begin{aligned}
E_{r, \mu}(u) & =\frac{1}{2}\|u\|^{2}-\Psi(u)-f_{r}\left(\frac{\norm{u}^2_2}{\mu}\right) \\
& \geq \frac{1}{2}\|u\|^{2}-\Psi(u)-f_{r}\left(\frac{\rho^{2}}{\mu{\hat{\lambda}_{1,1,1}}}\right) \\
& \geq \rho^{2}\left(\frac{1}{2}-\frac{K_2}{2{\hat{\lambda}_{1,1,1}}}-\frac{K_2^g}{2{\tilde{\lambda}_{1,1,1}}}-C \rho^{p-2}-C\rho^{l-2}-\frac{(\mu{\hat{\lambda}_{1,1,1}})^{-r} \rho^{2 r-2}}{1-(\mu{\hat{\lambda}_{1,1,1}})^{-1}\rho^{2}}\right).
\end{aligned}
$$
Therefore, it follows from conditions ($f_1'$) and ($g_1$) that Lemma \ref{lemmpg} holds.

Next, by applying Theorem \ref{thmpopen} and repeating the arguments in Lemma \ref{lembounded}-\ref{lemur} and \ref{lemcin}, we conclude that, there exists $u \in H^1(\Omega)$ satisfying $E(u)\leq \frac{\mu\hat{\lambda}_{1,1,1}}{2}$, and one of the following two cases must hold:
     \begin{enumerate}[label=\rm(\roman*)]
\item either $u$ is a critical point of $E$ constrained on $S_{1,1,1}(\mu)$ with Lagrange multiplier $\lambda \in [0,\hat{\lambda}_{1,1,1}]$.
\item  or $u$ is a critical point of $E$ constrained on $S_{1,1,1}(\nu)$ for some $0<\nu < \mu$ with Lagrange multiplier $\lambda =0$.
    \end{enumerate}
To rule out the case (ii), we establish the following lemma.
\begin{lemma}\label{lemnonnomass}
   Assume that $(f_1')$, $(f_4)$, ($g_1$) and ($g_2$) hold. Then, given $M>0$, there exists a constant $\mu^{**}:=\mu^{**}(\Omega,f,g,M)>0$ such that, there exists no $u \in H^1(\Omega)$ satisfying
   $$
   E'(u) = 0, \,\,\,\,\,E(u)\leq M\mu^{**}.
   $$
\end{lemma}
\begin{proof}
By \eqref{eqarq}, we have
    $$
    \norm{u}^2 \leq \frac{2qE(u)-2\langle E'(u),u\rangle}{q-2}=\frac{2qM\mu^{**}}{q-2}.
    $$
Thereby, by $E'(u)=0$, \eqref{eqgnh1} and \eqref{eqsi}, one has
\begin{equation*}
        \begin{aligned}
        \norm{u}^2=\int_{\Omega}f(u)u\,dx+\int_{\Omega}g(u)u\,d\sigma&\leq K_2 \norm{u}_2^2+K_{p} C_{p,\Omega}\norm{u}^{{p}} +K^g_2\norm{u}_{L^2(\partial\Omega)}^2 + K_lC_{p,\Omega}'\norm{u}^l\\
        &\leq \frac{K_2}{\hat{\lambda}_{1,1,1}} \norm{u}^2 + \frac{K^g_2}{\tilde{\lambda}_{1,1,1}} \norm{u}^2 +K_{p} C_{p,\Omega}\norm{u}^{{p}} + K_lC_{p,\Omega}'\norm{u}^l.
        \end{aligned}
    \end{equation*}
    Moreover, we obtain
    \begin{equation*}
    \begin{aligned}
            1-\frac{K_2}{\hat{\lambda}_{1,1,1}}-\frac{K^g_2}{\tilde{\lambda}_{1,1,1}}&\leq K_{p} C_{p,\Omega}\norm{u}^{p-2} + K_lC_{p,\Omega}'\norm{u}^{l-2} \\
            &\leq K_{p} C_{p,\Omega}\left(\frac{2qM}{q-2}\right)^{p-2}(\mu^{**})^{p-2}+ K_lC_{p,\Omega}'\left(\frac{2qM}{q-2}\right)^{l-2}(\mu^{**})^{l-2}.
    \end{aligned}
    \end{equation*}
It follows from ($f_1'$) and ($g_1$) that, there exists a constant $\mu^{**}=\mu^{**}(\Omega,f,g,M)>0$ such that, there exists no $u \in H^1(\Omega)$ satisfying $E'(u) = 0$ and $E(u)\leq M\mu^{**}$.
\end{proof}
Applying Lemma \ref{lemnonnomass} with $M=\frac{\hat{\lambda}_{1,1,1}}{2}$, we complete the proof of Theorem \ref{th2.3}.\\

Now, we provide a sketch of the proof of Theorem \ref{th4.3}. By the Hilbert-Schmidt theorem (see, e.g., \cite[Theorem 7 in Appendix D]{evans}), there exists an increasing sequence of infinity distinct eigenvalues $0<\hat{\lambda}_{1,1,1}<\hat{\lambda}_{2}<\hat{\lambda}_{3}<\cdots$, i.e., for any $j\geq 2$, there exists $\varphi_j \in H^1(\Omega)\backslash\{0\}$ satisfying
$$
\int_\Omega \nabla \varphi_j \nabla v\, dx + \int_{\partial\Omega}\varphi_j v\, d\sigma = \lambda_j\int_{\Omega}\varphi_j v\, dx, \quad \forall v \in H^1(\Omega).
$$
 For each $j \geq 2$ and $r>1$, define
$$
\begin{aligned}
    &Y_j:=\operatorname{span}\{u \in H^1(\Omega):\text{For some }\lambda \leq \hat{\lambda}_{j}\text{ and for any }v \in H^1(\Omega),\,\int_\Omega \nabla u\nabla v\, dx + \int_{\partial\Omega}uv\, d\sigma = \lambda\int_{\Omega}uv\, dx \},\\
    & Z_j:=Y_j^\perp\cup\operatorname{span}\{\varphi_j\},
\end{aligned}
$$
and define
$$
B_{r,j}:=\{u\in Y_j  : \norm{u}\leq \rho_{r,j}\}, \quad N_{r,j}:=\{u\in Z_j : \norm{u}= \xi_{r,j}\},
$$
where $\rho_{r,j}>\xi_{r,j}>0$, and we set $\rho_{r,j}= \sqrt{\mu\hat{\lambda}_{j}}$. Repeating the argument in Section \ref{secp4},  to prove Theorem \ref{th4.3}, it suffices to show that, for any $j \geq 2$ and any $\varepsilon>0$, there exist $\xi_{\varepsilon,j}>0$ and $r_{\varepsilon,j}>1$, such that, for all $r\geq r_{\varepsilon,j}$ and $\xi_{r,j}=\xi_{\varepsilon,j}$, we have
$$ \widetilde{b_{r,j}}\geq\frac{1}{2}\mu(\hat{\lambda}_{j}-K_2-\frac{K_2^g\hat{\lambda}_j}{\tilde{\lambda}_{1,1,1}})-\frac{K_pC_{p,N}\mu^\frac{p}{2}\hat{\lambda}_{j}^\frac{p}{2}}{p}-\left(\frac{k-1}{k}\right)^\frac{l}{2}\frac{K_lC_{l,\Omega}'\mu^\frac{l}{2}\hat{\lambda}_{j}^\frac{l}{2}}{l}-\varepsilon,$$
and that $\hat{\lambda}_{j}-K_2-\frac{K_2^g\hat{\lambda}_j}{\tilde{\lambda}_{1,1,1}} \to +\infty$ as $j \to +\infty$.

For some $k>1$ and $k\approx +\infty$, define
    $$\xi_{\varepsilon,j}= \sqrt{\frac{k-1}{k}\mu\hat{\lambda}_{j}}.$$
Then, for any $u \in Z_j$ with $\norm{u}=\xi_{\varepsilon,j}$, we have
    $$
     \norm{u}^2_2 \leq \frac{\norm{u}^2}{\hat{\lambda}_{j}} = \frac{\mu(k-1)}{k} \quad \text{and} \quad \norm{u}^2_{L^2(\partial\Omega)}\leq\frac{\norm{u}^2}{\tilde{\lambda}_{1,1,1}} = \frac{\mu\hat{\lambda}_{j}(k-1)}{k\tilde{\lambda}_{1,1,1}}
    $$
and thus, by \eqref{eqgnh1} and \eqref{eqsi}, it follows that
$$
\begin{aligned}
    E_{r,\mu}(u)\geq &\frac{k-1}{k}\frac{\mu\hat{\lambda}_{j}}{2}-\frac{(k-1)}{k}\frac{K_2\mu}{2}-\frac{(k-1)}{k}\frac{K_2^g\mu\hat{\lambda}_j}{2\tilde{\lambda}_{1,1,1}}\\&
    -\left(\frac{k-1}{k}\right)^\frac{p}{2}\frac{K_pC_{p,\Omega}\mu^\frac{p}{2}\hat{\lambda}_{j}^\frac{p}{2}}{p}-\left(\frac{k-1}{k}\right)^\frac{l}{2}\frac{K_lC_{l,\Omega}'\mu^\frac{l}{2}\hat{\lambda}_{j}^\frac{l}{2}}{l}-k\left(\frac{k-1}{k}\right)^r\\
    \geq &\frac{(k-1)\mu}{2k}(\hat{\lambda}_{j}-K_2-\frac{K_2^g\hat{\lambda}_j}{\tilde{\lambda}_{1,1,1}})\\&
    -\left(\frac{k-1}{k}\right)^\frac{p}{2}\frac{K_pC_{p,\Omega}\mu^\frac{p}{2}\hat{\lambda}_{j}^\frac{p}{2}}{p}-\left(\frac{k-1}{k}\right)^\frac{l}{2}\frac{K_lC_{l,\Omega}'\mu^\frac{l}{2}\hat{\lambda}_{j}^\frac{l}{2}}{l}-k\left(\frac{k-1}{k}\right)^r.
\end{aligned}
$$
Therefore, for any $\varepsilon>0$, there exist $\xi_{\varepsilon,j}>0$ and $r_{\varepsilon,j}>1$, such that, for all $r\geq r_{\varepsilon,j}$ and $\xi_{r,j}=\xi_{\varepsilon,j}$, one has
$$
\widetilde{b_{r,j}}\geq\frac{1}{2}\mu(\hat{\lambda}_{j}-K_2-\frac{K_2^g\hat{\lambda}_j}{\tilde{\lambda}_{1,1,1}})-\frac{K_pC_{p,N}\mu^\frac{p}{2}\hat{\lambda}_{j}^\frac{p}{2}}{p}-\left(\frac{k-1}{k}\right)^\frac{l}{2}\frac{K_lC_{l,\Omega}'\mu^\frac{l}{2}\hat{\lambda}_{j}^\frac{l}{2}}{l}-\varepsilon.
$$
By ($f_1'$) and ($g_1$), we know that $\hat{\lambda}_{j}-K_2-\frac{K_2^g\hat{\lambda}_j}{\tilde{\lambda}_{1,1,1}} \to +\infty$ as $j \to +\infty$.
Repeating the argument used in the proof of Theorem \ref{th4} and applying Lemma \ref{lemnonnomass} with $M=\frac{\hat{\lambda}_{k_m}}{2}$, we conclude that, for any \linebreak $0<\mu\leq \min\{\mu^{**}(\Omega,f,g,\frac{\lambda_{k_m}}{2}), \widetilde{\mu_{k_m}}\}$,  the functional $E$ admits at least $m$ critical points $u_{k_1}, u_{k_2},\cdots,u_{k_m}$ constrained on $S_{1,1,1}(\mu)$, which complete the proof of Theorem \ref{th4.3}.

\section{The problem \texorpdfstring{$(P)^\mu_{0,1,0}$}{Lg}}\label{secp010}
In this section, we consider the  problem  below
$$
\left\{
\begin{array}{ll}
	-\Delta u = \lambda u + f(u)\quad & \text{in } \Omega, \\
	\frac{\partial u}{\partial \eta} = 0 & \text{on } \partial \Omega, \\
	\int_{\Omega}u^2\,dx = \mu
\end{array}
\right.
\leqno{(P)^{\mu}_{1,0,0}}
$$
which is equivalent to solve the problem below
$$
\left\{
\begin{array}{ll}
	-\Delta u +u= \lambda u + f(u)\quad & \text{in } \Omega, \\
	\frac{\partial u}{\partial \eta} = 0 & \text{on } \partial \Omega, \\
	\int_{\Omega} u^2\,dx = \mu.
\end{array}
\right.
\leqno{\tilde{(P)}^{\mu}_{1,0,0}}
$$
We define the norm on $H^1(\Omega)$ by
$$
\|u\|:=\left(\int_{\Omega}|\nabla u|^2\,dx+\int_{\Omega}|u|^2\,dx\right)^{\frac{1}{2}}.
$$

Let $E: H^1(\Omega) \to \mathbb{R}$ be the energy functional  associated with problem $(P)^\mu_{0,1,0}$ defined by
$$
E(u) := \frac{1}{2}\norm{u}^{2}-\Psi(u),
$$
where
$$
\Psi(u)=\int_\Omega  F(u)\,dx.
$$
By the Gagliardo-Nirenberg inequality, for any $p \in [2, 2^*)$,
\begin{equation}\label{eqgnn}
	\|u\|_p^p \leq C_{p,\Omega}''\|u\|_2^{(1 - \beta_p)p} \| u\|^{\beta_p p} \quad \forall u \in  H^1(\Omega),
\end{equation}
where $\beta_p = N \left( \frac{1}{2} - \frac{1}{p} \right)$ and $ C_{p,\Omega}'' > 0$ depends only on  $p$ and $\Omega$. Moreover, $H^1(\Omega)$ is compactly embedded in $L^p(\Omega)$ for all $2\leq p<2^*$. Have this in mind, all computations performed in the previous section can be adapted to the present setting with only minor modifications. Assume that  ($f_1''$), ($f_2$) and ($f_3$) hold or ($f_1''$) and ($f_4$) hold. Then, we obtain a solution $(u_{pe},\lambda_{pe})\in H^1(\Omega)\times [-1,-1+\lambda_{0,1,0}]$ to problem $(P)^{\mu}_{1,0,0}$. However,  it remains unknown whether the solution $u_{pe}$ is one of the constant functions $\pm\sqrt{\frac{\mu}{\abs{\Omega}}}$, because $\left(\pm\sqrt{\frac{\mu}{\abs{\Omega}}}, -\frac{f\left(\pm\sqrt{\frac{\mu}{\abs{\Omega}}}\right)}{\pm\sqrt{\frac{\mu}{\abs{\Omega}}}}\right)$ are also solutions to problem $(P)^{\mu}_{1,0,0}$.  An open and interesting question is whether one can distinguish  $u_{pe}$ from the constant functions. Let $f(u)=\abs{u}^{p-2}u$ with $2+\frac{4}{N}<p<2^*$ and $\mu >0$ small. Then, the constant function $\sqrt{\frac{\mu}{\abs{\Omega}}}$ is a local minimizer on $S_{0,1,0}(\mu):=\{u \in H^1(\Omega): \int_{\Omega}\abs{u}^2\, dx=\mu\}$, see \cite[Remark 3.6]{Chang}. Thus, the constant function $\sqrt{\frac{\mu}{\abs{\Omega}}}$ is a solution to equation
$$
-\Delta u=-\left(\frac{\mu}{\abs{\Omega}}\right)^\frac{p-2}{2} u+\abs{u}^{p-2}u
$$ with the Morse index $m\left(\sqrt{\frac{\mu}{\abs{\Omega}}}\right)=1$. In \cite[Remark 3.6]{Chang}, the authors proves problem $(P)^{\mu}_{1,0,0}$ has a solution $(u_{mp},\lambda_{mp})\in H^1(\Omega)\times \R$ and $u$ is a critical point of mountain pass type constrained on $S_{0,1,0}(\mu)$. Then, $u_{mp}$ cannot be a constant function. However, this argument fails when we apply the perturbation method. In fact, repeating the argument of Lemma \ref{lemunr} and \ref{lemur}, we have $\norm{u_{r_n}}^2_2<\mu=\left\Vert\sqrt{\frac{\mu}{\abs{\Omega}}}\right\Vert_2^2$ and $\displaystyle \lim_{n \to +\infty}u_{r_n}=u_{pe}$, but the limit behavior of $u_{r_n}$ remains unknown. Moreover, by a standard argument, see \cite{Chang,Bu}, we will conclude that $u_{pe}$ is a solution to equation
$$
-\Delta u=\lambda_{pe} u+\abs{u}^{p-2}u
$$ with the Morse index $m(u_{pe})=1=m\left(\sqrt{\frac{\mu}{\abs{\Omega}}}\right)$, hence, it is hard to distinguish $u_{pe}$ from the constant functions with a Morse index argument.

Now suppose further that $f(t)=-f(-t)$. Then, following the argument in the proof of Theorem \ref{th4}, we conclude that, for any $m \in \mathbb{N}^+$, there exists $\mu_{m}^{***} > 0$, depending on $\Omega$ and $f$, such that, for any $0 < \mu < \mu^{***}_{m}$,  $(P)^{\mu}_{0,1,0}$ admits at least $m$ nontrivial solutions $(u_1,\bar{\lambda}_1), (u_2,\bar{\lambda}_2),\cdots,(u_m,\bar{\lambda}_m) \in H^{1}(\Omega)\times[0,+\infty)$, and these solutions have distinct energy levels, i.e., $E(u_i) \neq E(u_j)$, for any $1\leq i\neq j \leq m$. Since there exist exactly two constant solutions of the form $u_{\pm c}=\pm\sqrt{\frac{\mu}{\abs{\Omega}}}$, with associated Lagrange multiplier $\lambda_{\pm c}=-\frac{f(u)}{u}$, and since both yield the same energy value, i.e. $$
E(\sqrt{\frac{\mu}{\abs{\Omega}}})=E(-\sqrt{\frac{\mu}{\abs{\Omega}}}),
$$
without loss of generality, we can assume that, for all $1\leq i\leq m-1$, $u_i$ is not a constant function.
\section{Existence of ground states to problem  \texorpdfstring{$(P)^\mu_{1,0,0}$}{Lg}}\label{secgs}

We first characterize the energy of $u$, where $(u,\lambda_u)$ is a solution to problem $(P)^{\mu}_{1,0,0}$ obtained in Theorems \ref{th2} or \ref{th3}.

Define
$$
\mathcal{S}^+(\mu):=\{u \in S_{1,0,0}(\mu):\langle E'(u),\cdot \rangle=\lambda_u (u,\cdot)_2\text{ for some } \lambda_u \geq 0 \},
$$
and
\begin{equation}\label{eqdefn+}
    \mathcal{N}^+(\mu):=\{u \in S_{1,0,0}(\mu):\langle E'(u),u\rangle \geq 0\}.
\end{equation}
\begin{lemma}\label{lemminima}
   Assume that either ($f_1$)-($f_3$) hold or ($f_1$), ($f_2$) and ($f_4$) hold. Let $(u,\lambda_u)$ denote the solution of problem $(P)^{\mu}_{1,0,0}$ obtained in Theorems \ref{th2} or \ref{th3}. Then, $u$ satisfies
    \begin{equation}\label{eqN+=}
        E(u)= \inf_{u \in \mathcal{S}^+(\mu)}E(u)= \inf_{u \in \mathcal{N}^+(\mu)}E(u).
    \end{equation}
\end{lemma}
\begin{proof}
Let $(u,\lambda_u)$ be a solution to problem $(P)^{\mu}_{1,0,0}$ obtained in Theorem \ref{th2} and \ref{th3}. It is clear that $u \in \mathcal{S}^+(\mu)$ and $\mathcal{S}^+(\mu)\subset \mathcal{N}^+(\mu)$, hence,
\begin{equation}\label{eqN+1}
E(u)\geq \inf_{u \in \mathcal{S}^+(\mu)}E(u)\geq \inf_{u \in \mathcal{N}^+(\mu)}E(u).
\end{equation}
We now prove
\begin{equation}\label{eqN+2}
    E(u) \leq  \inf_{u \in \mathcal{N}^+(\mu)}E(u).
\end{equation}
 For any $v \in \mathcal{N}^+(\mu)$, define
 $$
 g(t):= E(tv),\quad t\geq 0.
 $$
Then, by (f$_2$), for all $0\leq t \leq 1$, we have
$$
g'(tv) = t \norm{v}^2-\int_\Omega f(tv)v\,dx \geq  t\left(\norm{v}^2-\int_\Omega f(v)v\,dx\right)=t\langle E'(v),v\rangle\geq 0,
$$
which implies $g(t)$ is non-decreasing on $[0,1]$.
By \eqref{eqcssup}, we conclude that
$$
E(u)=c_\infty\leq \sup _{0 \leq t<1} E(t v) =\sup _{0 \leq t<1} g(t) = \max _{0 \leq t\leq 1}g(t)=g(1)=E(v).
$$
This proves \eqref{eqN+2}. Thus, by \eqref{eqN+1}, we obtain
\begin{equation*}
    c_\infty=E(u)= \inf_{u \in \mathcal{S}^+(\mu)}E(u)= \inf_{u \in \mathcal{N}^+(\mu)}E(u).
\end{equation*}
\end{proof}

We note that \eqref{eqN+=} implies that $u$ is a minimizer of $E$ constrained on $\mathcal{N}^+(\mu)$,  and that $(u,\lambda_u)$ is a solution to  problem $(P)^{\mu}_{1,0,0}$ such that its energy $E(u)$ is minimal among all the solutions $(v,\lambda_v)$ of problem $(P)^{\mu}_{1,0,0}$ with $\lambda_v \geq 0$. Moreover, this result can be easily extended to all other cases mentioned in this paper.


Define
$$
\mathcal{S}(\mu):=\{u \in S_{1,0,0}(\mu):\langle E'(u),\cdot \rangle=\lambda_u (u,\cdot)_2\text{ for some } \lambda_u \in \R \}.
$$
\begin{definition}\label{defgs}
     We say that $(u, \lambda_u)$ is a normalized ground state solution to problem $(P)^{\mu}_{1,0,0}$, if $u \in \mathcal{S}(\mu)$ such that
     $$E(u) = \inf_{u \in \mathcal{S}(\mu)}E(u).$$
\end{definition}
Next, we present our main result of this section, which concerns the existence of ground states to problem $(P)^{\mu}_{1,0,0}$.
\begin{theorem}\label{thgs}
    Let $\Omega\subset \mathbb{R}^N(N \geq 3)$ be a smooth bounded  domain,  and let $\Omega$ be star-shaped with respect to the origin. Suppose that ($f_1$), ($f_2$) and ($f_4$) hold with $q > 2+\frac{4}{N}$.
If $0<\mu \leq \mu^*_s$, where $\mu^*_s$ is defined in \eqref{eqmus}, then $(P)^{\mu}_{1,0,0}$ has a ground state solution $(u,\lambda_u) \in H^1_0(\Omega) \times \left[\frac{2\lambda_{1,0,0}(q-2^*)}{2^*(q-2-\frac{4}{N})},\lambda_{1,0,0}\right]$.
\end{theorem}

The following lemma plays a key role in proving the existence of ground states for problem $(P)^{\mu}_{1,0,0}$.
\begin{lemma}\label{lemlambdau}
Suppose that ($f_1$), ($f_2$) and ($f_4$) hold with $q > 2+\frac{4}{N}$. For all $u \in \mathcal{S}(\mu)$, when $E(u) \leq \frac{\mu\lambda_{1,0,0}}{2}$, we have $$\frac{N\lambda_{1,0,0}(q-2^*)}{2^*(q-2)} \leq \lambda_u \leq \lambda_{1,0,0}.$$
\end{lemma}
\begin{proof}
    Since $u \in \mathcal{S}(\mu)$, we have
    \begin{equation}\label{equins}
            \norm{\nabla u}^2_2=\lambda_u\norm{u}^2_2 + \int_{\Omega}f(u)u\, dx=\mu\lambda_u + \int_{\Omega}f(u)u\, dx.
    \end{equation}
    By $E(u) \leq \frac{\mu\lambda_{1,0,0}}{2}$ and ($f_4$), we obtain
    $$
    \begin{aligned}
            \frac{\mu\lambda_{1,0,0}}{2} \geq E(u)&=\frac{1}{2}\norm{\nabla u}^2_2- \int_{\Omega}F(u)\, dx\\&=\frac{\mu\lambda_u}{2} + \frac{1}{2}\int_{\Omega}f(u)u\, dx- \int_{\Omega}F(u)\, dx \\&\geq \frac{\mu\lambda_u}{2} + (\frac{1}{2}-\frac{1}{q})\int_{\Omega}f(u)u\, dx,
    \end{aligned}
    $$
hence, $\lambda_u \leq \lambda_1$ and
\begin{equation}\label{eqfuu}
    \int_{\Omega}f(u)u\, dx \leq \frac{q\mu  (\lambda_{1,0,0}-\lambda_u)}{q-2}.
\end{equation}
By the Pohoz\v{a}ev identity, see, e.g., \cite[Theorem B.1]{W}, we have
$$
\frac{N-2}{2N} \int_{\Omega}\abs{\nabla u}^2\, dx+ \frac{1}{2N}\int_{\partial\Omega}\abs{\nabla u}^2(x\cdot\bm{n})\, d\sigma = \frac{\lambda_u}{2}\int_{\Omega}\abs{u}^2\, dx + \int_{\Omega}F(u)\, dx,
$$
where $\bm{n}(x)$ denotes the outward unit normal vector at $x \in \partial\Omega$. Since $\Omega$ be star-shaped with respect to the origin, we have $x\cdot \bm{n}(x)>0$. Thus,
$$
\frac{N-2}{2N} \int_{\Omega}\abs{\nabla u}^2\, dx\leq  \frac{\lambda_u}{2}\int_{\Omega}\abs{u}^2\, dx + \int_{\Omega}F(u)\, dx=\frac{\mu\lambda_u}{2} + \int_{\Omega}F(u)\, dx.
$$
By \eqref{equins}, we obtain
$$
\frac{\mu\lambda_u}{2}+  \int_{\Omega}F(u)\, dx \geq \frac{(N-2)\mu\lambda_u}{2N} + \frac{N-2}{2N} \int_{\Omega}f(u)u\, dx.
$$
Therefore, by \eqref{eqfuu} and ($f_4$),
$$
\lambda_u\geq  \frac{N}{\mu}\int_{\Omega}\left(\frac{1}{2^*}f(u)u-F(u)\right)\, dx \geq \frac{N}{\mu}\left(\frac{1}{2^*}-\frac{1}{q}\right)\int_{\Omega}f(u)u\, dx \geq \frac{N(\lambda_{1,0,0}-\lambda_u)(q-2^*)}{2^*(q-2)},
$$
hence, since $q > 2+\frac{4}{N}$, we conclude that
$$
\lambda_u \geq \frac{2\lambda_{1,0,0}(q-2^*)}{2^*(q-2-\frac{4}{N})}.
$$
\end{proof}
It is the position to provide the proof of Theorem \ref{thgs}.
\begin{proof}[Proof of Theorem \ref{thgs}]
Set $s := \frac{2\lambda_{1,0,0}(q-2^*)}{2^*(q-2-\frac{4}{N})}>0$. Then, we can define an equivalent norm on $H_0^1(\Omega)$ by
$$
\|u\|_s:=\left(\int_{\Omega}|\nabla u|^2\,dx+s \int_{ \Omega}|u|^2\,dx\right)^{\frac{1}{2}}.
$$ Define the energy functional $E_s : H_0^1(\Omega)\to \R$ by
$$
E_s(u):=\frac{1}{2}\norm{\nabla u}_2^2 +\frac{s}{2}\norm{u}_2^2-\int_{\Omega}F(u)\, dx.
$$
From ($f_4$),  we know that for any $R>0$ such that $F(R)>0$, one gets
$$
\quad F(t)\geq \frac{F(R)}{R^q} \abs{t}^q, \quad \forall t\in (-\infty,-R)\cup(R,+\infty),
$$
which implies $p \geq q >2+\frac{4}{N}$. Set
\begin{equation}\label{eqmus}
    \mu^*_s:=
    \left(\frac{\lambda_{1,0,0}+s-K_2}{K_pC_{p,N}}\right)^{\frac{2}{p-2}}\left(\frac{q}{q-2}\right)^{{\frac{2}{p-2}}-\frac{N}{2}}(\lambda_{1,0,0}+s)^{-\frac{N}{2}}.
\end{equation}
Note that, for all $u \in H_0^1(\Omega)$, we have $\norm{u}^2_s \geq (\lambda_{1,0,0}+s)\norm{u}^2_2$ and $\norm{u}^2_s \geq \norm{\nabla u}^2_2$.
By repeating the argument in the proof of Theorem \ref{th3}, we conclude that, for any $0<\mu\leq \mu^*_s$, there exists $u\in H_0^1(\Omega) $ satisfying $E_s(u) \leq \frac{\mu(\lambda_{1,0,0}+s)}{2}$ and $u$ is a critical point of $E_s$ constrained on $S_{1,0,0}(\mu)$ with a Lagrange multiplier $\lambda_{s,u} \in [0,\lambda_{1,0,0}+s]$, i.e., $u$ satisfies $E(u) \leq \frac{\mu\lambda_{1,0,0}}{2}$ and $u \in \mathcal{S}(\mu)$ with $\lambda_u \in [-s,\lambda_{1,0,0}]$.

By repeating the argument in Lemma \ref{lemminima}, we know that $u$ satisfies
$$E(u)= \inf_{u \in \mathcal{S}_s^+(\mu)}E(u),$$
where
$$
\begin{aligned}
    \mathcal{S}_s^+(\mu):=&\{u \in S_{1,0,0}(\mu):\langle E_s'(u),\cdot \rangle=\lambda_{s,u} (u,\cdot)_2\text{ for some } \lambda_{s,u} \geq 0 \} \\=& \{u \in S_{1,0,0}(\mu):\langle E'(u),\cdot \rangle=(\lambda_{s,u}-s) (u,\cdot)_2\text{ for some } \lambda_{s,u} \geq 0 \}
    \\=&\{u \in S_{1,0,0}(\mu):\langle E'(u),\cdot \rangle=\lambda_u (u,\cdot)_2\text{ for some } \lambda_{u} \geq -s \}.
\end{aligned}
$$
It is clear that $\displaystyle \inf_{u \in \mathcal{S}(\mu)}E(u)\leq E(u)\leq \frac{\mu\lambda_{1,0,0}}{2}$. Therefore, by Lemma \ref{lemlambdau}, we conclude that
$$
\inf_{u \in\mathcal{S}_s^+(\mu)}E(u) =E(u)\geq \inf_{u \in \mathcal{S}(\mu)}E(u)=\inf_{\begin{subarray}{c}u \in \mathcal{S}(\mu),\\E(u)\leq \frac{\mu\lambda_{1,0,0}}{2}\end{subarray}}E(u)\geq \inf_{u \in\mathcal{S}_s^+(\mu)}E(u).
$$
This completes the proof of Theorem \ref{thgs}.
\end{proof}
Let $s' > \frac{2\lambda_{1,0,0}(q-2^*)}{2^*(q-2-\frac{4}{N})}>0$ and
$$
\mu^*_{s'}:=
    \left(\frac{\lambda_{1,0,0}+s'-K_2}{K_pC_{p,N}}\right)^{\frac{2}{p-2}}\left(\frac{q}{q-2}\right)^{{\frac{2}{p-2}}-\frac{N}{2}}(\lambda_{1,0,0}+s')^{-\frac{N}{2}}.
$$
 Then, by repeating the argument in Theorem \ref{thgs} and Lemma \ref{lemminima}, for $0<\mu\leq\mu_{s'}^*$, we know that $(P)^{\mu}_{1,0,0}$ has a ground state solution $(u_{s'},\lambda_{u_{s'}}) \in H^1_0(\Omega) \times \left[-s',\lambda_{1,0,0}\right]$ and $u$ satisfies
 $$E(u)= \inf_{u \in \mathcal{N}_{s'}^+(\mu)}E(u),$$
where
$$
    \mathcal{N}_{s'}^+(\mu)=\{u \in S_{1,0,0}(\mu):\langle E'(u),u \rangle\geq -s'(u,u)_2=-s'\mu\}.
$$
By Lemma \ref{lemlambdau}, we conclude that $\lambda_{u_{s'}}>-s'$, hence $u_{s'}$ is a minimizer of $E$ constrained on $$ \tilde{\mathcal{N}}_{s'}^+(\mu):=\{u \in S_{1,0,0}(\mu):\langle E'(u),u \rangle>-s'\mu\},$$
which implies that $u_{s'}$ is a local minimizer of $E$ constrained on $S_{1,0,0}(\mu)$.

By repeating the argument in Corollary \ref{co1}, we have
\begin{corollary}\label{co2}
Let $\Omega\subset \mathbb{R}^N(N \geq 3)$ be a smooth bounded  domain, and let $\Omega$ be star-shaped with respect to the origin. For any fixed $\mu>0$, the following two existence results depending on $\Omega$ hold:
  \begin{enumerate}[label=\rm(\roman*)]
    \item Suppose that ($f_1$), ($f_2$) and ($f_4$) hold with $q > 2+\frac{4}{N}$, and suppose that $K_2=0$ and $2+\frac{4}{N}<p<2^*$. If $\Omega$ satisfies $0<\lambda_{1,0,0}\leq \frac{q-2}{q}\left(1+\frac{2(q-2^*)}{2^*(q-2-\frac{4}{N})}\right)^{-1}\mu^{\frac{2(p-2)}{4+2N-Np}}(K_pC_{p,N})^
    \frac{4}{4+2N-Np}$, then $(P)^{\mu}_{1,0,0}$ has a ground state solution $(u,\lambda) \in H^1_0(\Omega)  \times \left[\frac{2\lambda_{1,0,0}(q-2^*)}{\cdot2^*(q-2-\frac{4}{N})},\lambda_{1,0,0}\right]$.
    \item Suppose that ($f_2$) and ($f_4$) hold with $q > 2+\frac{4}{N}$, and suppose $|f(t)|\leq a\left(\abs{t}^{p'-1} +\abs{t}^{p-1}\right)$ with $2+\frac{4}{N}<p'<p<2^*$ for some positive constant $a>0$. There exists $\lambda^*$ depending on $a$, $p$, $p'$, $q$, $N$ and $\mu$ such that, if $\Omega$ satisfies $0<\lambda_{1,0,0}<\lambda^*$, then $(P)^{\mu}_{1,0,0}$  has a ground state solution $(u,\lambda) \in H^1_0(\Omega)  \times \left[\frac{(q-2^*)\lambda_{1,0,0}}{2\cdot2^*(q-2-\frac{4}{N})},\lambda_{1,0,0}\right]$.
\end{enumerate}
\end{corollary}
An open and interesting question is whether one can distinguish the ground state solution from the solution obtained in Theorem \ref{th3}.

\section{Further Applications}\label{secfinalremarks}
In this section, we highlight further applications of our perturbation method, including the nonlinear Schr\"{o}dinger equations with critical exponential growth in $\mathbb{R}^{2}$, the nonlinear Schr\"{o}dinger equations with magnetic fields, the biharmonic equations, and the Choquard  equations, and others. Since the proofs of these results are similar to those presented earlier,  we only state the main conclusions here and omit the detailed proofs.
\subsection{The nonlinear Schr\"{o}dinger equation with exponential critical growth in \texorpdfstring{$\Omega\subset\mathbb{R}^2$}{Lg};} \mbox{}\\
Assume that $f$ is a continuous function with  exponential critical growth. That is,
there exists $\alpha_0>0$ such that
\begin{itemize}
	\item[$(f_1)$] $\displaystyle \lim_{\abs{t} \to \infty}\frac{\abs{f(t)}}{e^{\xi t^2}}=0, \,\,\text{for all}\,\, \xi >\alpha_0 $ \quad
	and \quad $\displaystyle \lim_{\abs{t} \to \infty}\frac{\abs{f(t)}}{e^{\xi t^2}}=+\infty, \,\,\text{for all}\,\,  \xi <\alpha_0. $
\end{itemize}
In addition, we assume that $f$ satisfies the following conditions:
\begin{itemize}
	\item[$(f_2)$] There exists $\theta>2$ such that
	$$
0\leq \theta F(t) \leq f(t)t,
	$$
	where $F(t)=\int_{0}^{t}f(s)\,ds$.
    \item[$(f_3)$] There exist two constants $p>2$ and $c_p>0$ such that
    $$
    \abs{f(t)} \geq c_p\abs{t}^{p-1},\,\,\text{for all}\,\, t\in \mathbb{R}.
    $$
\end{itemize}
Then, the problem
\begin{equation} \label{exponetial}
\left\{
\begin{array}{ll}
	-\Delta u = \lambda u + f(u), \quad & \text{in } \Omega, \\
	u = 0, & \text{on } \partial \Omega, \\
	\int_{\Omega} u^2\,dx = \mu,
\end{array}
\right.
\end{equation}
where $\Omega \subset \mathbb{R}^2$ is a smooth bounded domain, admits a solution for $\mu$ small enough and $c_p>0$ large.  The largeness of $c_p$ ensures good control of the $H_{0}^{1}(\Omega)$-norm of $(PS)$  sequences,  which is essential for applying the  Trudinger-Moser inequality to overcome the lack of compactness due to the critical exponential growth.

Set $c_p^*>0$ large enough. Then, we can obtain the following existence and multiplicity of normalized solutions for $\mu>0$ is small enough.
\begin{theorem}
Suppose that ($f_1$)-($f_3$) hold and that $c_p>c_p^*$. Then, there exists a constant $\mu^{*}>0$ depending only on $\Omega$ and $f$ such that, for all $0<\mu<\mu^{*}$, problem \eqref{exponetial} admits a solution $(u,\lambda) \in H^1_0(\Omega) \times  [0, {\lambda}_{1,1,1}]$.
\end{theorem}
\begin{theorem}
Suppose that ($f_1$)-($f_3$) hold and that $c_p>c_p^*$ large, and suppose that $f(t)=-f(-t)$. For any $m \in \mathbb{N}^+$, there exists $\mu^{*}_{m} > 0$ depending on $\Omega$ and $f$ such that, for any $0 < \mu < \mu^{*}_{m}$, problem \eqref{exponetial}  admits at least $m$ nontrivial solutions $(u_1,\bar{\lambda}_1), (u_2,\bar{\lambda}_2),\cdots,(u_m,\bar{\lambda}_m) \in H_0^{1}(\Omega)\times[0,+\infty)$.
\end{theorem}
\subsection{The nonlinear Schr\"{o}dinger equations with magnetic fields} \mbox{}\\
We also investigate the existence and multiplicity of normalized solutions to the following nonlinear Schr\"{o}dinger equations with magnetic fields
\begin{equation} \label{P5}
	\left\{
	\begin{aligned}
		&\left( -i \nabla -A(x)\right)^2u =\lambda u+f(|u|^2)u \quad  \hbox{in} \quad \Omega,\\
		& u=0 \quad  \hbox{on} \quad \partial \Omega,\\
		&\int_{\Omega}|u|^{2}dx=\mu,
	\end{aligned}
	\right.
\end{equation}
where $\Omega \subset \mathbb{R}^N$ ($N \geq 3$)  is a smooth bounded domain,   and $A\in C(\overline{\Omega}, \mathbb{R}^N)$ is a continuous magnetic potential.

In the  $3$-dimensional case, the magnetic field $B$  is given by the curl of $A$, i.e. $B=\nabla \times A$. For higher dimensions $N \geq 4$, the magnetic field is represented by the 2-form $B_{i,j} = \partial_j A_k - \partial_k A_j$.

We suppose that the nonlinearity $f\in C(\R^+,\R)$  satisfies  the following assumptions:
\begin{itemize}
	\item[$(f_4)$] $f(0)=0$,   and there exist constants $c_1> 0$ and $p \in (2,2^*)$ such that
	$$
	\limsup_{s \to +\infty}	\frac{|f(s)|}{|s|^{\frac{p-2}{2}}} <+\infty ,
	$$
	where $2^*=\frac{2N}{N-2}$ if $N \geq 3$.
\item[$(f_5)$] There exists $q>2$ such that
$$
0 \leq \frac{q}{2}F(t) \leq f(t)t,\,\, \,\text{for all}\,\,\,t > 0,
$$
where $F(t)=\int_{0}^{t}f(s)\,ds$.
\end{itemize}
Under these assumptions, we establish the existence and multiplicity of normalized solutions for $\mu>0$ small enough.
\begin{theorem}
Suppose that ($f_4$)-($f_5$) hold. Then, there exists a constant $\mu^{*}>0$  depending only on $\Omega$, $f$ such that, for all $0<\mu<\mu^{*}$, problem \eqref{exponetial} has a solution $(u,\lambda) \in H^1_0(\Omega,\C) \times  [0,+\infty)$.
\end{theorem}
\begin{theorem}
Suppose that ($f_4$)-($f_5$) hold, and suppose that $f(t)=-f(-t)$. For any $m \in \mathbb{N}^+$, there exists $\mu^{*}_{m} > 0$ depending on $\Omega$, $f$ such that, for any $0 < \mu < \mu^{*}_{m}$, problem \eqref{exponetial}  has at least $m$ nontrivial solutions $(u_1,\bar{\lambda}_1), (u_2,\bar{\lambda}_2),\cdots,(u_m,\bar{\lambda}_m) \in H_0^{1}(\Omega,\C)\times[0,+\infty)$.
\end{theorem}

\subsection{Bi-harmonic equations} \mbox{} \\
Another interesting class of problem that our approach can be applied is the bi-harmonic equation of the form
\begin{equation} \label{eqbi1}
	\left\{
	\begin{array}{ll}
		\Delta^2 u-\beta \Delta u = \lambda u + f(u),\,\,& \text{in}\,\, \Omega, \\
		u =\Delta u= 0 \,\, (\mbox{or}\,\, u=\frac{\partial u}{\partial \eta}=0 ), & \text{on } \partial \Omega, \\
		\int_{\Omega} u^2\,dx = \mu,
	\end{array}
	\right.
\end{equation}
where $\Omega \subset \mathbb{R}^N$ ($N\geq 1$) is a smooth bounded domain and $\beta \in \mathbb{R}$. Under suitable assumptions, this problem admits a solution for $\mu>0$ small enough. For example, we may assume $f \in C(\R,\R)$ and satisfies:
\begin{itemize}
	\item[$(f_6)$] There exists $p \in (2,2_*)$ such that
	$$
	\limsup_{|s| \to +\infty}\frac{|f(s)|}{|s|^{p-1}}<+\infty,
	$$
where $2_* = \frac{2N}{N-4}$ if $N \geq 5$ and $2_* = +\infty$ if $N = 1, 2, 3, 4$.
	\item[$(f_7)$] There exists $\theta>2$ such that
	$$
	0 \leq \theta F(t) \leq f(t)t,
	$$
	where $F(t)=\int_{0}^{t}f(s)\,ds$.
\end{itemize}

It is well known that the space $H^2_0(\Omega)$ (for the Dirichlet boundary condition $u =\Delta u= 0$) or $H^2(\Omega) \cap H^1_0(\Omega)$ (for the Navier boundary condition $u=\frac{\partial u}{\partial \eta}=0$) is a Hilbert space equipped with the norm
$$
\norm{\Delta u}^2_2=\int_\Omega \abs{\Delta u}^2\,dx.
$$
By \cite{Ha2}, we have
\begin{equation}\label{eqbiharmonic}
    \norm{\nabla u}_2^2 \leq C_B\norm{\Delta u}_2\norm{u}_2.
\end{equation}
If $\beta \geq 0$, then all computations in the previous sections apply here with minor modifications.
Now, suppose that $\beta < 0$. Applying  Young's inequality together with \eqref{eqbiharmonic}, we have
$$
\norm{\Delta u}_2^2-\beta \norm{\nabla u}_2^2\geq  \norm{\Delta u}_2^2-\beta C_B\norm{\Delta u}_2\norm{u}_2\geq \norm{\Delta u}_2^2 - \frac{1}{2} \norm{\Delta u}_2^2 -2\beta^2 C_B^2\norm{u}^2_2.
$$
Let $s:= 1+ 2\beta^2 C_B^2$, it follows that
$$
\norm{\Delta u}_2^2-\beta \norm{\nabla u}_2^2 + s\norm{u}^2_2 \geq \frac{1}{2}\norm{\Delta u}_2^2+ \norm{u}^2_2.
$$
Then, with slight modifications, all the computation  in the previous section can be adapted to the following equivalent formulation of problem \eqref{eqbi1}.
\begin{equation} \label{eqbi2}
	\left\{
	\begin{array}{ll}
		\Delta^2 u-\beta \Delta u +su = \lambda u + f(u), \quad & \text{in } \Omega, \\
		u =\Delta u= 0 \,\, ( \mbox{or} \,\, u=\frac{\partial u}{\partial \eta}=0 ), & \text{on } \partial \Omega, \\
		\int_{\Omega} u^2\,dx = \mu.
	\end{array}
	\right.
\end{equation}

More precisely, we can obtain the following existence and multiplicity of normalized solutions for $\mu>0$ is small enough.
\begin{theorem}
Suppose that ($f_6$)-($f_7$) hold. Then, there exists $\mu^{*}>0$ (or $\mu^{**} > 0$)  depending on $\Omega$, $f$ such that, for all $0<\mu<\mu^{*}$ (or $0<\mu<\mu^{**}$), problem \eqref{eqbi1} has a solution $(u,\lambda) \in H^2_0(\Omega) \times  \R$ for the Dirichlet boundary condition $u =\Delta u= 0$ (or $(u,\lambda) \in (H^2(\Omega) \cap H^1_0(\Omega)) \times \R$ for the Navier boundary condition $u=\frac{\partial u}{\partial \eta}=0$).
\end{theorem}
\begin{theorem}
Suppose that ($f_6$)-($f_7$) hold, and suppose that $f(t)=-f(-t)$. For any $m \in \mathbb{N}^+$, there exists $\mu^{*}_{m} > 0$ (or $\mu^{**}_{m} > 0$) depending on $\Omega$, $f$ such that, for any $0 < \mu < \mu^{*}_{m}$ (or $0<\mu<\mu^{**}_m$), problem \eqref{eqbi1}  has at least $m$ nontrivial solutions $(u_1,\bar{\lambda}_1), (u_2,\bar{\lambda}_2),\cdots,(u_m,\bar{\lambda}_m) \in H_0^{2}(\Omega)\times  \R$ for the Dirichlet boundary condition $u =\Delta u= 0$ (or $(u_1,\bar{\lambda}_1), (u_2,\bar{\lambda}_2),\cdots,(u_m,\bar{\lambda}_m)\in (H^2(\Omega) \cap H^1_0(\Omega)) \times \R$ for the Navier boundary condition $u=\frac{\partial u}{\partial \eta}=0$).
\end{theorem}
\subsection{Choquard equation} \mbox{}\\
Our approach can also be applied to the study of normalized solutions for the following class of Choquard equations:
\begin{equation} \label{choquard}
	\left\{
	\begin{array}{ll}
		-\Delta u = \lambda u + \left(\int_{\Omega}\frac{|u(y)|^{p}}{|x-y|^{\alpha}} \, dx\right)|u|^{p-2}u, \quad & \text{in } \Omega, \\
		u = 0, & \text{on } \partial \Omega, \\
		\int_{\Omega} u^2\,dx = \mu,
	\end{array}
	\right.
\end{equation}
where $\Omega \subset \mathbb{R}^N$ ($N\geq 1$) is a smooth bounded domain, $\alpha \in (0,N)$ and $p$ satisfies
$$
2-\frac{\alpha}{N}<p<\frac{2^*}{2}\left(2-\frac{\alpha}{N} \right).
$$
where $2^* = \frac{2N}{N-2}$ if $N \geq 3$ and $2^* = +\infty$ if $N = 1, 2$.

Under these assumptions, the existence and multiplicity of normalized solutions can be established for $\mu>0$ small enough.
\begin{theorem}
Suppose that $\alpha \in (0,N)$ and $p$ satisfies
$
2-\frac{\alpha}{N}<p<\frac{2^*}{2}\left(2-\frac{\alpha}{N} \right)
$. Then, there exists $\mu^{*}>0$  depending on $\Omega$, $f$ such that, for all $0<\mu<\mu^{*}$, problem \eqref{choquard} has a solution $(u,\lambda) \in H^1_0(\Omega) \times  [0,\lambda_{1,1,1})$.
\end{theorem}
\begin{theorem}
Suppose that $\alpha \in (0,N)$ and $p$ satisfies
$
2-\frac{\alpha}{N}<p<\frac{2^*}{2}\left(2-\frac{\alpha}{N} \right)
$. Then, there exists $\mu^{*}_{m} > 0$ depending only on $\Omega$, $f$ such that, for any $0 < \mu < \mu^{*}_{m}$, problem \eqref{choquard}  has at least $m$ nontrivial solutions $(u_1,\bar{\lambda}_1), (u_2,\bar{\lambda}_2),\cdots,(u_m,\bar{\lambda}_m) \in H_0^{1}(\Omega)\times[0,+\infty)$.
\end{theorem}

 \section{appendix} \label{appa}
    In this appendix, we present a mountain pass type theorem for $C^1$ functional defined on open subset, due to \cite{Bu,Es} that was used in Section \ref{secp100}.

    Let $X$ be a Banach space, and let $U$ be a open subset of $X$ satisfying $\partial U \neq \emptyset $. For $u \in X$ and $R>0$, define $B_R(u):=\{v \in X: \norm{v-u}_X<R\}$. Let $I \in C^1(U,\R)$ and assume that for each $u \in \partial U$, there exists $\varepsilon:=\varepsilon(u)>0$ such that  $I(v)\leq -1$  for all $v \in U\cap B_\varepsilon(u)$.

    Let $\beta \in C^\infty(\mathbb{R},\mathbb{R})$ be such that
$$
\begin{cases}
    \beta \equiv -1 &\text{ on } (-\infty,-1),\\
    \beta (t)=t, &\forall t\geq 0,\\
    \beta (t)\leq 0,  &\forall t\leq  0.
\end{cases}
$$
Define a new functional $J: X \to \mathbb{R}$ by
$$J(u) =
\begin{cases}
    \beta(I(u)) &\text{ if } u \in U,\\
    -1  &otherwise.
\end{cases}
$$
\begin{lemma}\label{lemc1}
    $J \in C^1(X,\R)$ and
    $$J'(u) =
\begin{cases}
    \beta'(I(u))I'(u) &\text{ if } u \in U,\\
    0  &\text{ if } u \in U^c.
    \end{cases}
$$
Moreover, if $J(u)>0$, then $u \in U$ and $J'(u) =\beta'(I(u))I'(u)=I'(u)$.
\end{lemma}
\begin{proof}
    It is clear that $J \in C^1(U,\R)\cup C^1(\bar{U}^c,\R)$ with
$$
J'(u) =
\begin{cases}
    \beta'(I(u))I'(u) &\text{ if } u \in U,\\
    0  &\text{ if } u \in \bar{U}^c.
    \end{cases}
$$
and, if $J(u)>0$, then $u \in U$ and $J'(u) =\beta'(I(u))I'(u)=I'(u)$.
Hence, to prove Lemma \ref{lemc1}, it suffices to show that for any $u \in \partial U$, $J'(u)$ exits in a sense of Fr\'echet derivative. Here, we are going to show that  $J'(u)=0$ when  $u \in \partial U$, from where it follows $J'$ is continuous in $X$. By the definition of $I$ and $J$, we conclude that, for any $v \in U\cap B_\varepsilon(u)$, we have $J(v)=\beta(I(v))=-1=J(u)$ and $J'(v)=\beta'(I(v))I'(v)=0$. Moreover, since, for any $v \in U^c\cap B_\varepsilon(u)$, $J(v)=-1=J(u)$, we obtain $J'(u)=0$. Since $u \in \partial U$ is arbitrary, for any $v \in \partial U$, we conclude $J'(v)=0$, and thus, for any $v \in  U\cap B_\varepsilon(u)$, we have $J'(v)=0$, which completes the proof of Lemma \ref{lemc1}.
\end{proof}
\begin{theorem}\label{thmpopen}
 Let $X$, $U$, $I$ and $J$ be defined as above and further assume that $0 \in U$ and $I(0)=0$. If $$
c:=\inf_{\gamma \in \Gamma} \max _{t \in[0,1]} I(\gamma(t))>0,
$$
where
$$
\Gamma := \{\gamma \in C([0,1], U): \gamma(0)=0,\,  I(\gamma(1))<0\}.
$$
Then, there exists a Cerami sequence of $I$ at level $c$, i.e., a sequence $\{u_n\}\subset U$ satisfying
$$
I(u_{n}) \rightarrow c \text{ and } (1+\norm{u_{n}})\norm{I^{\prime}(u_{n})} \rightarrow 0 \quad \text{ as } n \to +\infty.
$$
\end{theorem}
\begin{proof}
Define
$$
c':=\inf_{\gamma \in \Gamma'} \max _{t \in[0,1]} J(\gamma(t)),
$$
and
$$
\Gamma':= \{\gamma \in C([0,1], X): \gamma(0)=0,\,  J(\gamma(1))<0\}.
$$

Firstly, we show that $c= c'$. It follows from $J(u)=\beta(I(u))=u$ for all $u \in U$ with $I(u)>0$ that, for all $\gamma \in \Gamma$, we have $\max _{t \in[0,1]} J(\gamma(t))=\max _{t \in[0,1]} I(\gamma(t))  \geq c>0$. Since $\Gamma \subset \Gamma'$, we conclude that $c \geq c'$. On the other hand, by the definition of $c'$, for any $\varepsilon>0$, there exists a $\gamma_\varepsilon' \in \Gamma'$ such that $c'\leq\max _{t \in[0,1]} J(\gamma(t))< c'+\varepsilon$. It is clear that, if $\gamma_\varepsilon'([0,1])\cap U^c = \emptyset$, then $\gamma_\varepsilon' \in \Gamma$, and thus, $c < c'+\varepsilon$. As a result, we may assume that $\gamma_\varepsilon'([0,1])\cap U^c \neq \emptyset$. Let $$t_0:= \min\{t \in [0,1]:\gamma_\varepsilon'(t)\in U^c\}.$$
Since $B_{\rho}(0)\subset U$ for $\rho>0$ small enough, we know that $t_0 >0$. Then, by the definition of $t_0$, we know that, for all $t \in [0,t_0)$, $\gamma_\varepsilon'(t) \in U$. Thus, $\gamma_\varepsilon'(t_0) \in\partial U$. By the definition of $I$, there exists $\delta>0$ such that, for all $t \in (t_0-\delta,t_0)$, $I(\gamma_\varepsilon'(t))\leq -1<0$. Let $t_1 \in (t_0-\delta,t_0)$, and define
$$
\begin{aligned}
    \gamma_\varepsilon: [0,1]&\to U\\
     t &\mapsto \gamma_\varepsilon'\left(\frac{t}{t_1}\right).
\end{aligned}
$$
Then, $\gamma_\varepsilon \in \Gamma$, and
$$
c\leq \max _{t \in[0,1]} I(\gamma_\varepsilon(t))=\max _{t \in[0,1]} J(\gamma_\varepsilon(t))=\max _{t \in[0,t_1]} J(\gamma_\varepsilon'(t))\leq \max _{t \in[0,1]} J(\gamma_\varepsilon'(t)) < c'+\varepsilon.
$$
Since $\varepsilon>0$ is arbitrary, we conclude that $c=c'$.

In a view of \cite[Theorem 2.9 and 4.2]{W} and \cite{Bar,Ce}, since $c'=c>0$, we obtain a Cerami sequence of $J$ at level $c$, i.e., a sequence $\{u_n\}\subset X$ satisfying
$$
J(u_{n}) \rightarrow c \text{ and } (1+\norm{u_{n}})\norm{J^{\prime}(u_{n})} \rightarrow 0 \quad \text{ as } n \to +\infty.
$$
By Lemma \ref{lemc1}, we conclude that $u_n \in U$ for $n$ sufficiently large, and, up to a subsequence, we have
$$
I(u_{n}) \rightarrow c \text{ and } (1+\norm{u_{n}})\norm{I^{\prime}(u_{n})} \rightarrow 0 \quad \text{ as } n \to +\infty.
$$
\end{proof}
\subsection*{Conflict of interest}

On behalf of all authors, the corresponding author states that there is no conflict of interest.

\subsection*{Ethics approval}
 Not applicable.

\subsection*{Data Availability Statements}
Data sharing not applicable to this article as no datasets were generated or analysed during the current study.

\subsection*{Acknowledgements}
The authors would like to thank Rui Ding for helpful discussions and suggestions. C. Ji is supported by National Natural Science Foundation of China (No. 12171152).

  \end{document}